\documentclass{amsproc}
\usepackage{epsfig,amscd,amssymb,amsmath,amsfonts}
\newtheorem{theorem}{Theorem}[section]
\newtheorem{lemma}[theorem]{Lemma}

\theoremstyle{definition}

\newtheorem{example}[theorem]{Example}

\theoremstyle{remark}
\newtheorem{remark}[theorem]{Remark}
 \newcommand{\CC}{\mathbb{C}}

\numberwithin{equation}{section}

\begin{document}

\title{Remarks on 2-dimensional HQFT's }

\author{Mihai D. Staic}
\thanks{The work of M.\ Staic was  partially supported by the
CNCSIS project ''Hopf algebras, cyclic homology and monoidal
categories'' contract nr. 560/2009. The work of V.\ Turaev was
partially supported by the NSF
  grants DMS-0707078 and DMS-0904262. }
\address{Department of Mathematics, Indiana University, Rawles Hall, Bloomington, IN 47405, USA}
\address{Institute of Mathematics of the Romanian Academy, PO.BOX 1-764, RO-70700, Bu\-cha\-rest, Romania}
%\curraddr{}

\email{mstaic@indiana.edu}
%\curraddr{}


\author{Vladimir Turaev}
\address{Department of Mathematics, Indiana University, Rawles Hall, Bloomington, IN 47405, USA }
\email{vtouraev@indiana.edu}

%    General info
%\subjclass{Primary ...}
\date{January 1, 1994 and, in revised form, June 22, 1994.}

%\dedicatory{I thank to   }


 \begin{abstract}
 We introduce and study algebraic structures
 underlying  2-dimen\-si\-onal Homotopy Quantum Field Theories (HQFTs)
 with   arbitrary target spaces. These algebraic structures are formalized in the notion of a twisted Frobenius algebra. Our work
 generalizes
results of Brightwell,  Turner, and the second author on
2-dimen\-si\-onal HQFTs with simply-connected or aspherical targets.
 \end{abstract}

\maketitle

%\section*{This is an unnumbered first-level section head}
%
%$$\begin{CD}
%A @>a>> B\\
%@VVbV @VVcV\\
%C @>d>> D
%\end{CD}$$

%%%%%%%%%%%%%%%%%%%%%%%%%%%%%%%%%%%%%%%%%%%%%%%%%%%%

\section*{Introduction}

A   fruitful idea  in topology is to construct invariants of
manifolds that behave functorially with respect to   gluings of
manifolds along   the  boundary. More formally, one associates to
every closed oriented $d$-dimensional manifold~$M$ a finite
dimensional vector space ${V}_M$ and to every compact oriented
$(d+1)$-dimensional cobordism $(W,M,N)$ a homomorphism $\tau_W:{V}_M
\to {V}_N$. The conditions satisfied by $({V},\tau)$ in order to
obtain a consistent theory are formalized in the notion of a
$(d+1)$-dimensional Topological Quantum Field Theory (TQFT),  see
\cite{am}. The second author introduced more general   Homotopy
Quantum Field Theories (HQFT's), see \cite{tu}. One can think of a
$(d+1)$-dimensional HQFT as of a TQFT for $d$-dimensional manifolds
and $(d+1)$-dimensional cobordisms endowed with maps to a fixed
space~$X$.

The 1-dimensional HQFTs with target $X$ are easily classified in
terms  of finite dimensional representations of $\pi_1(X)$. A study
of 2-dimensional HQFTs is more involved and leads to interesting
algebra. For contractible $X$,  the category of 2-dimensional HQFTs
with target   $X$ is equivalent to the category of  2-dimensional
TQFTs and  is known to be equivalent to the category of commutative
finite-dimensional  Frobenius algebras. If $X=K(G,1)$ is an
Eilenberg-MacLane space determined by a group $G$, then   the
category of 2-dimensional HQFTs with target   $X$ is equivalent to
the category of so-called crossed Frobenius $G$-algebras, see
\cite{tu}. If $X=K(A,2)$ is an Eilenberg-MacLane space determined by
an abelian group $A$, then the category of 2-dimensional HQFTs with
target   $X$ is equivalent to the category of Frobenius
$A$-algebras, see \cite{bt}.

In this paper we address the case where both groups $G=\pi_1(X)$ and
$A=\pi_2(X)$ are allowed to be  nontrivial. An important role   will
be  played by  the first $k$-invariant $k  \in H^3(G, A)$ of $X$. We
shall introduce and study algebraic structures
 underlying  2-dimensional HQFTs with target $X$. We call them twisted Frobenius algebras or, shorter,
  TF-algebras over  $(G, A, k)$.
 Briefly speaking, a TF-algebra is   a $G$-graded algebra of
 $A$-modules
which  is associative   up to a 3-cocycle representing $k$ and which
satisfies a form of commutativity. This work  is a    step towards
   classification of 2-dimensional HQFTs with target $X$
(for a different approach, see~\cite{pt}).

The paper is organized as follows.  We introduce TF-algebras in
Section~\ref{frof}. In Section~2 we recall the notion of an HQFT and
the definition of the first $k$-invariant of a topological space.
In Section~3   we derive from any 2-dimensional HQFT the
  underlying  TF-algebra.   In Section~4 we compute the underlying TF-algebras of  the
   cohomological 2-dimensional HQFTs.

Throughout the paper the symbol $K$ denotes a field and $\otimes=\otimes_K$.

\section{Twisted Frobenius Algebras}\label{frof}

\subsection{Preliminaries}\label{preli}
In this section   $G$ is a group with neutral element $\varepsilon$
and  $A$ is a left $G$-module with neutral element $1=1_A$.  We
use multiplicative notation for the group operation in $A$.  As
usual, the group ring of $A$ with coefficients in $K$ is denoted
$K[A]$. The action of $\alpha \in G$ on $a\in A$ is denoted by
$\,^\alpha a$.

To calculate the cohomology   $H^*(G,A)$, we use  the standard
cochain complex
$$C^0(G,A)\stackrel{\delta^0}{\to} C^1(G,A) \stackrel{\delta^1}{\to} C^2(G,A)\stackrel{\delta^2}{\to} C^3(G,A)
\stackrel{\delta^3}{\to}  C^4(G,A)\to  \cdots $$
where   $C^n(G,A)=Map(G^n,A) $ for any $n\geq 0$. For small $n$, the
coboundary homomorphism  $\delta^n:C^n(G,A)\to C^{n+1}(G,A)$  is
given by the following formulas:
\begin{eqnarray*}
\delta^0(a)(\alpha )=\,^\alpha a \, a^{-1}
\end{eqnarray*}
for any $a\in A=C^0(G, A)$ and $\alpha\in G$,
\begin{eqnarray*}
\delta^1(\psi)(\alpha,\beta)=\,^\alpha\psi(\beta)\, \psi(\alpha\beta)^{-1}\psi(\alpha)
\end{eqnarray*}
for any $\psi\in  C^1(G, A)$ and $\alpha, \beta\in G$,
\begin{eqnarray*}
\delta^2(\omega)(\alpha,\beta,\gamma)=\,^\alpha\omega(\beta,\gamma)\, \omega(\alpha\beta,\gamma)^{-1}\omega(\alpha,\beta\gamma)\, \omega(\alpha,\beta)^{-1}
\end{eqnarray*}
for any $\omega\in  C^2(G, A)$ and $\alpha, \beta, \gamma\in G$,
\begin{eqnarray*}
\delta^3(\chi)(\alpha,\beta,\gamma,\rho)=\,^\alpha\chi(\beta,\gamma,\rho)\, \chi(\alpha\beta,\gamma,\rho)^{-1}\chi(\alpha,\beta\gamma,\rho)\, \chi(\alpha,\beta,\gamma\rho)^{-1}\chi(\alpha,\beta,\gamma)
\end{eqnarray*}
for any $\chi\in  C^3(G, A)$ and $\alpha, \beta, \gamma, \rho\in G$.

\subsection{Definition of TF-algebras}  Fix from now on a normalized 3-cocycle  $\kappa:G^3\to A$. Thus,
 for all $\alpha, \beta, \gamma, \rho\in G$,
\begin{eqnarray*}
\,^\alpha\kappa(\beta,\gamma,\rho)\,\kappa(\alpha\beta,\gamma,\rho)^{-1}\,\kappa(\alpha,\beta\gamma,\rho)\,\kappa(\alpha,\beta,\gamma\rho)^{-1}\,\kappa(\alpha,\beta,\gamma)=1.
\end{eqnarray*}
The word \lq\lq normalized" means  that for all $\alpha, \beta \in
G$,
\begin{eqnarray*}
\kappa(\varepsilon,\alpha,\beta)=\kappa(\alpha,\varepsilon,\beta )=\kappa(\alpha,\beta, \varepsilon)=1.
\end{eqnarray*}

%%%\begin{definition}
A {\it twisted Frobenius algebra (TF-algebra)} over the triple
$(G,A,\kappa)$ is a $G$-graded $K[A]$-module $V=\oplus_{\alpha \in
G} V_{\alpha}$ such that

1) ({\it Underlying module and action of $A$.}) The underlying
$K$-vector space  of $V_{\alpha}$ is    finite-dimensional  and for
all $\alpha \in G$, $u\in V_{\alpha}$, and $a\in A$,
\begin{eqnarray}
au=\,^{\alpha}au . \label{bimodule}
\end{eqnarray}

2) ({\it Multiplication.}) We have a $K $-bilinear multiplication
$V\times V\to V$ carrying $ V_{\alpha}\times V_{\beta}$ to $
V_{\alpha\beta}$ for all $\alpha, \beta\in G$. For any $u\in
V_{\alpha}$, $v\in V_{\beta}$, $w\in V_{\gamma}$ with $\alpha,
\beta, \gamma \in G$,
\begin{eqnarray}
(uv)w=\kappa(\alpha,\beta,\gamma)\,u(vw). \label{associativity}
\end{eqnarray}
There is  a unit element  $1_V\in V_{\varepsilon}$ such that $1_V
u=u1_V=u$ for all $u\in V $.

3) ({\it Inner product.}) We have   a non-degenerate symmetric
$K$-bilinear form $$\eta_{\varepsilon}: V_{\varepsilon} \otimes
V_{\varepsilon}\to K$$ such that for all $\alpha\in G$, the pairing
\begin{eqnarray}
V_{\alpha} \otimes V_{\alpha^{-1}}\to K, \,\,  u\otimes v \mapsto  \eta_{\varepsilon}(uv\otimes 1_V) \label{nas}
\end{eqnarray}
(where $u\in V_{\alpha}$, $v\in V_{\alpha^{-1}}$) is non-degenerate.

%%%For every $\alpha\in G$ we  have a nondegenerate form $\eta_{\alpha}: V_{\alpha} \otimes V_{\alpha^{-1}}\to K$ such that
%%%if $a\in A$, $u\in V_{\alpha}$, and $v\in V_{\alpha^{-1}}$ then:
%%%\begin{eqnarray}
%%%\eta_{\alpha}(u\otimes v)= \eta_{\varepsilon}(uv\otimes 1).\label{nas}
%%%\end{eqnarray}
%%%If $u, v\in V_{\varepsilon}$:
%%%\begin{eqnarray}
%%%\eta_{\varepsilon}(u\otimes v)=\eta_{\varepsilon}(v\otimes u).\label{es}
%%%\end{eqnarray}

4) ({\it Projective action of $G$.}) For each $  \beta \in G$, we
have a $K$-linear isomorphism $\varphi_{\beta}:V\to V$ carrying
$V_{\alpha}$ to $V_{\beta\alpha\beta^{-1}}$ for all $\alpha$ and
such that
\begin{eqnarray}
\varphi_{\beta}(av)=\,^{\beta}a\,\varphi_{\beta}(v)  \,\, {\text {for any}}\,\,   a\in A   \,\, {\text {and}}\,\,
v\in V,
\label{mor}
\end{eqnarray}
\begin{eqnarray}
%\varphi_{\beta \vert  V_{\beta}}=k(\alpha,e,\alpha)id_{V_{\alpha}}
\varphi_{\beta} \vert_{V_{\beta}} ={\text {id}}_{V_{\beta}}:V_{\beta} \to V_{\beta}, \label{id}
\end{eqnarray}
\begin{eqnarray}
\varphi_{\beta}(1_V)=1_V ,
\label{mor+}
\end{eqnarray}
\begin{eqnarray}
vu=\varphi_{\beta}(u)v \,\, {\text {for any}}\,\,   u\in V    ,
v\in   V_{\beta}, \label{com}
\end{eqnarray}
\begin{eqnarray}
\eta_{\varepsilon}(\varphi_{\beta}(u)\otimes \varphi_{\beta}(v)) =\eta_{\varepsilon}(u\otimes v)
\,\, {\text {for any}}\,\,   u,v\in V_{\varepsilon},\label{epg}
\end{eqnarray}
\begin{eqnarray} \varphi_{\beta}(u)\, \varphi_{\beta}(v)=l_{\alpha,\gamma}^{\beta}\, \varphi_{\beta}(uv)
\,\, {\text {for any}}\,\,  {\alpha}, {\gamma} \in G \,\, {\text {and}}\,\,  u\in  V_{\alpha}, v\in V_{\gamma}   ,
\label{pd}
\end{eqnarray}
where
\begin{eqnarray*}
l_{\alpha,\gamma}^{\beta}=\kappa( {\beta}\alpha {\beta}^{-1},  {\beta}\gamma {\beta}^{-1},
\beta)\,\kappa( {\beta} \alpha {\beta}^{-1},\beta,\gamma)^{-1}\,\kappa(\beta,\alpha,\gamma)\in A,
%\label{khat}
\end{eqnarray*}
\begin{eqnarray}
\varphi_{\gamma\beta}\vert_{V_{\alpha}}=h_{\gamma, \beta}^{\alpha}\,(\varphi_{\gamma}\circ \varphi_{\beta})\vert_{V_{\alpha}}
\,\, {\text {for all}}\,\,   \alpha,\beta,\gamma\in G,
\label{pgb}
\end{eqnarray}
where
\begin{eqnarray*}
h_{\gamma, \beta}^{\alpha}= \kappa(\gamma\beta\alpha(\gamma\beta)^{-1},\gamma,\beta)\,
\kappa(\gamma,\beta\alpha\beta^{-1},\beta)^{-1}\kappa(\gamma,\beta,\alpha)\in A.
%\label{ktilde}
\end{eqnarray*}

5) ({\it The trace condition.}) For any $\alpha, \beta\in G$ and
$c\in V_{\alpha\beta\alpha^{-1}\beta^{-1}}$,
\begin{eqnarray}
\begin{aligned}
{\text {Tr}} (\kappa(\alpha\beta\alpha^{-1}\beta^{-1},\beta,\alpha)\, \mu_c\, \varphi_{\beta}:V_{\alpha}\to V_{\alpha})\\
={\text {Tr}} (\kappa(\alpha\beta\alpha^{-1}\beta^{-1},\beta\alpha\beta^{-1},\beta)
\,\varphi_{\alpha}^{-1}\mu_c:V_{\beta}\to V_{\beta}),\label{tr}
\end{aligned}
\end{eqnarray}
where $\mu_c $ is left  multiplication $V\to V, v\mapsto cv$ by $c$
and ${\text {Tr}}$ is the standard trace of linear maps.
%\begin{eqnarray}
%Tr(\mu_c\varphi_{\beta}:V_{\alpha}\to V_{\alpha})
%=Tr(t_{\alpha,\beta} \varphi_{\alpha}^{-1}\mu_c:V_{\beta}\to V_{\beta})\label{tr}
%\end{eqnarray}
%where
%\begin{eqnarray*}
%t_{\alpha,\beta}=\kappa(\alpha\beta\alpha^{-1}\beta^{-1},\beta\alpha\beta^{-1},\beta)\kappa(\alpha\beta\alpha^{-1}\beta^{-1},\beta,\alpha)^{-1}\in A
%\end{eqnarray*}
%%%\end{definition}

 We have the following elementary consequences of the definition. The $K[A]$-bilinearity of multiplication in $ V$ implies that
\begin{eqnarray}
a(uv)=(au)v=u(av)  \label{multabil}
\end{eqnarray}
for all $a\in A$ and $u,v\in V$. Note that if $u\in V_\alpha$ and
$v\in V_\beta$, then for all $a\in A$, $$^{\alpha}a (u v)= (^{\alpha
 }a u) v =  ( a u) v=a(uv)$$
 and similarly $^{\beta}a (u  v)=  a(uv)$.  Formula
\eqref{pgb} applied to $\gamma=\beta=\varepsilon$ implies that
\begin{eqnarray}
\varphi_{\varepsilon}={\text {id}}:V\to V.  \label{idid}
\end{eqnarray}

In the following lemma and in the sequel the pairing \eqref{nas} is
denoted by $\eta_\alpha$.

\begin{lemma}
For any $a\in A$, $u\in V_{\alpha}$, $v\in V_{\alpha^{-1}}$ with
$\alpha\in G$,
\begin{eqnarray*}
\eta_{\alpha}(au\otimes v)=\eta_{\alpha}(u\otimes av),
\end{eqnarray*}
\begin{eqnarray*}
\eta_{\alpha}(u\otimes v)=\eta_{\alpha^{-1}}(\kappa(\alpha^{-1},\alpha,\alpha^{-1})^{-1}\, v\otimes u).
\end{eqnarray*}
For any $\alpha, \beta\in G$, $u\in V_{\alpha}$, $v\in
V_{\alpha^{-1}}$,
\begin{eqnarray*}
\eta_{\beta\alpha\beta^{-1}}(\varphi_{\beta}(u)\otimes \varphi_{\beta}(v))=\eta_{\alpha}(l_{\alpha,\alpha^{-1}}^{\beta} u\otimes v).
\end{eqnarray*}
For any $\alpha, \beta\in G$, $u\in V_{\alpha}$, $v\in V_{\beta}$
and $w\in V_{(\alpha\beta)^{-1}}$,
\begin{eqnarray*}
\eta_{\alpha\beta}(uv\otimes w)=\eta_{\alpha}(\kappa(\alpha,\beta,(\alpha\beta)^{-1})\, u\otimes vw).
\end{eqnarray*}
\end{lemma}
\begin{proof}
 We check only the first two identities  leaving the other two to the reader.
 We have
\begin{eqnarray*}
\eta_{\alpha}(au\otimes v)&=&\eta_{\varepsilon}((au)v\otimes 1_V)\\
&=&\eta_{\varepsilon}(u(av)\otimes 1_V)\\
&=&\eta_{\alpha}(u\otimes av).
\end{eqnarray*}
Formulas  (\ref{id}), \eqref{idid}, and  (\ref{pgb}) with $\gamma=\alpha$, $\beta=\alpha^{-1}$ imply  the identity
$\varphi_{\alpha}(v)=
\kappa(\alpha^{-1},\alpha,\alpha^{-1})^{-1}\,v$ for all $v\in
V_{\alpha^{-1}}$. Therefore
\begin{eqnarray*}
\eta_{\alpha}(u\otimes v)&=&\eta_{\varepsilon}(uv\otimes 1_V)\\
&=&\eta_{\varepsilon}(\varphi_{\alpha}(v)u\otimes 1_V)\\
&=&\eta_{\varepsilon}( \kappa(\alpha^{-1},\alpha,\alpha^{-1})^{-1} \,v u\otimes 1_V)\\
&=&\eta_{\alpha^{-1}}(\kappa(\alpha^{-1},\alpha,\alpha^{-1})^{-1}  \, v\otimes u).
\end{eqnarray*}
\end{proof}

Given   $z\in K^*$ and  a TF-algebra, we can multiply  the inner
product   $\eta_{\varepsilon}$ by $z $
 (keeping the rest of the data) and obtain thus a new TF-algebra. This operation on TF-algebras is
called {\it  $z$-rescaling}.

  Let $V$,  $W$ be TF-algebras over   $(G,A,\kappa)$.  A $K[A]$-isomorphism
 $f:V\to W$ is an {\it isomorphism} of TF-algebras if $f$ is an isomorphism of algebras
 such that  $
\eta_{\varepsilon}(f(u)\otimes f(v))=\eta_{\varepsilon}(u\otimes v)
$ for all $u,v\in V_{\varepsilon}$  and $ f \varphi_{\beta}
=\varphi_{\beta} f$ for all $\beta\in G$.

\subsection{Examples}
The definition of a TF-algebra over   $(G,A,\kappa)$ generalizes
both the notion of a crossed Frobenius $G$-algebra  \cite{tu}  and
the notion of a  $A$-Frobenius algebra   \cite{bt}. Consequently we
have the following  two sources of examples.
\begin{example}
If  $L$ is a crossed Frobenius  $G$-algebra,   then $L$ is a
TF-algebra over $(G,A,\kappa)$  where $A$ is the trivial group and
$\kappa$ is the trivial cocycle.
\end{example}
\begin{example} If $V$ is a $A$-Frobenius algebra,
  then $V$ is a TF-algebra over $(G,A,\kappa)$ where $G$ is the trivial group and $\kappa$ is the trivial cocycle.
\end{example}
Further  examples of TF-algebras are   constructed   in
Section~\ref{uTFC}.

\subsection{Coboundary equivalence}\label{kprime} Let   $\kappa:G^3\to A$ be a normalized 3-cocycle.
Given a normalized 2-cochain
$\omega:G^2\to
 A$, the map $\kappa'=  \delta^2(\omega) \, \kappa:G^3\to A$ is a normalized 3-cocycle    cohomological to $\kappa$.
 Using $\omega$, we can transform a  TF-algebra $V$ over $(G,A,\kappa)$ into a TF-algebra $V^{\omega}$
 over $(G,A,\kappa')$  as follows. The
underlying $G$-graded $K[A]$-modules of $V^{\omega}$ and $V$ are the
same. The inner product   on $V^{\omega}_\varepsilon=V_\varepsilon$
is the same as in $V$. Multiplication $\cdot^{\omega}$ on
$V^{\omega}$ is defined by
\begin{eqnarray*}
u\cdot^{\omega} v=\omega(\alpha,\beta)^{-1}\, u\cdot v,
\end{eqnarray*}
for  $u \in V_\alpha, v\in V_\beta$, where $\cdot$ is multiplication
in $V$. Given $\beta\in G$, the   automorphism
$\varphi^{\omega}_\beta$ of $V^{\omega}$ is defined by
\begin{eqnarray*}
\varphi^{\omega}_\beta \vert_{V_\alpha}=\omega(\beta,\alpha)^{-1}\, \omega(\beta\alpha\beta^{-1},\beta) \,
\varphi_{\beta}\vert_{V_\alpha}
\end{eqnarray*} for all $\alpha \in G$.
Direct computations show that $V^\omega$ is a TF-algebra
 over $(G,A,\kappa')$.   We   say
that $V^{\omega}$ is obtained from $V$ by a
  {\it coboundary transformation}.
 This transformation  defines an equivalence
 between the category of TF-algebras  over $(G,A,\kappa)$ and their isomorphisms and
the category of TF-algebras  over $(G,A,\kappa')$ and their
isomorphisms.

Given $k\in H^3(G,A)$, the coboundary transformations and the
isomorphisms generate an equivalence relation on the class of
TF-algebras over the triples $(G,A,\kappa)$ where $\kappa:G^3\to A$
runs over all normalized 3-cocycles representing $k$. This relation
is called {\it coboundary equivalence} and the corresponding
equivalence classes   are called TF-algebras over $(G,A,k)$.

\section{Preliminaries on HQFTs and $k$-invariants}

\subsection{HQFTs}
We recall the definition of an HQFT from \cite{tu},  \cite{tu2}. We say that a topological space is {\it pointed}
 if all its connected components are provided with base points.
By   {\it maps} of pointed spaces we mean continuous maps carrying base points to base points.

Fix   a pointed path connected topological space $X$. An {\it $X$-manifold} is a pair $(M,g)$,
 where $M$ is a pointed closed oriented manifold and $g$ is a   map $M\to X$.
  An {\it $X$-homeomorphism}  between $(M,g)$ and $(M',g')$ is an orientation preserving (and base point preserving)
   homeomorphism $f:M\to M'$ such that $g=g'f$. An empty set is  considered as an $X$-manifold of any given dimension.

An {\it $X$-cobordism}  between $X$-manifolds $(M_0,g_0)$ and $(M_1,g_1)$ is an oriented compact
 cobordism $(W,M_0,M_1)$ endowed with a map $g:W\to X$
  such that $\partial W=(-M_0) \amalg M_1$ and  $g |_{M_i}=g_i$ for $i=1,2$. Note that $W$ is not required to be pointed, but $M_0$ and $M_1$ are pointed.
  An $X$-homeomorphism   between two $X$-cobordisms
  $(W,M_0,M_1,g)$ and $(W',M_0',M_1',g')$   is an orientation preserving
   homeomorphism of triples
   $F:(W,M_0,M_1)\to (W',M_0',M_1')$ such that $g=g'F$ and $F$ restricts to $X$-homeomorphisms
    $M_0\to M'_0$ and $M_1\to M'_1$. For example, for any    $X$-manifold $(M,g)$ we have
{\it the cylinder cobordism} $ (M\times [0,1],M\times 0,M\times 1,\overline{g}) $ between two copies of $(M,g)$. Here
$\overline{g}$ is the composition of the projection $M\times [0,1]\to M$  with $g$. Any $X$-homeomorphism
of $X$-manifolds multiplied by ${\text {id}}_{[0,1]}$ yields   an $X$-homeomorphism of the corresponding
cylinder cobordisms.

  A {\it (d+1)-dimensional Homotopy Quantum Field Theory $({V},\tau)$
 with target $X$} assigns a   finite-dimensional $K$-vector space ${V}_M$ to any $d$-dimensional $X$-manifold, a $K$-isomorphism $f_{\#}:{V}_M\to {V}_{M'}$ to any $X$-homeomorphism of $d$-dimensional $X$-manifolds $f:M\to M'$ and a $K$-homomorphism $\tau(W):{V}_{M_0}\to {V}_{M_1}$ to any $(d+1)$-dimensional $X$-cobordism $(W,M_0,M_1)$.
 These vector spaces  and homomorphisms should satisfy the  following axioms:

(1) For any $X$-homeomorphisms of $d$-dimensional $X$-manifolds $f:M\to M'$, $f':M'\to M''$,
 we have $(f'f)_{\#}=f'_{\#}f_{\#}$.

 %%%The isomorphism $f_{\#}:{V}_{M}\to {V}_{M'}$ is
 %%%invariant under isotopies of $f$ in the class of $X$-homeomorphisms.

(2) For any disjoint $d$-dimensional $X$-manifolds $M$, $N$, there
is a natural isomorphism ${V}_{M\amalg  N}={V}_M\otimes {V}_N$.

(3) ${V}_{\emptyset}=K$.

(4) For any $X$-homeomorphism of $(d+1)$-dimensional $X$-cobordisms
$$F:(W,M_0,M_1,g) \to  (W', M_0', M_1', g'),$$
the following diagram is commutative:
$$\begin{CD}
{V}_{(M_0,g|_{M_0})} @> (F|_{M_0})_{\#} >> {V}_{(M'_0,g'|_{M'_0})}\\
@VV\tau(W,g)V @VV\tau(W',g')V\\
{V}_{(M_1,g|_{M_1})}@>(F|_{M_1})_{\#}>> {V}_{(M'_1,g'|_{M'_1})}
\end{CD}$$

(5) If a $(d+1)$-dimensional $X$-cobordism $W$ is a disjoint union
of $X$-cobordisms $W_1$, $W_2$, then $\tau(W)=\tau(W_1)\otimes
\tau(W_2)$.

(6) If an $X$-cobordism $(W,M_0,M_1)$ is obtained from two
$(d+1)$-dimensional $X$-cobordisms $(W_0,M_0,N)$ and $(W_1,N',M_1)$
by gluing along an $X$-homeomorphism $f:N\to N'$, then
$$\tau(W)=\tau(W_1)f_{\#}\tau(W_0):{V}_{M_0}\to {V}_{M_1}.$$

(7) For any $d$-dimensional $X$-manifold $(M,g)$,
$$\tau(M\times [0,1],M\times 0,M\times 1,\overline{g})= {\text {id}}:{V}_M\to {V}_M, $$
where we identify $M\times 0$ and $M\times 1$ with $M$ in the obvious way  and the triple
 $(M\times [0,1],M\times 0,M\times 1,\overline{g})$ is the cylinder cobordism over $(M,g)$.

(8) For any $(d+1)$-dimensional $X$-cobordism $(W,M_0,M_1,g)$, the
homomorphism $\tau(W)$ is preserved under homotopies of $g$ constant
on $\partial W=M_0\amalg  M_1$.

\begin{remark} In this definition we follow \cite{tu2} rather than \cite{tu}. The difference is that axiom (7) in \cite{tu2} is weakened in comparison with the corresponding axiom in \cite{tu}.
\end{remark}

For shortness, HQFTs with target $X$ will be also called $X$-HQFTs.
For examples   of $X$-HQFTs, see \cite{tu},  \cite{tu2}. Note here
that any cohomology class   $\theta\in H^{d+1}(X;K^*)$ determines a
$(d+1)$-dimensional $X$-HQFT $({V}^\theta, \tau^\theta)$.

An {\it isomorphism of $X$-HQFTs} $\rho:({V},\tau)\to ({V}',\tau')$
is a family of $K$-i\-so\-mor\-phisms $\{\rho_M : {V}_M \to
{V}_M'\}_M$, where $M$ runs over  all $d$-dimensional $X$-manifolds,
  $\rho_{\emptyset}= {\text {id}}_K$,   $\rho_{M\amalg  N} = \rho_{M}\otimes \rho_{N}$ for all $M, N$,
 and the
 natural square diagrams associated with the $X$-homeomorphisms and $X$-cobordisms are commutative.

\subsection{The   $k$-invariant}\label{XXX} Let   $X$ be a   path connected topological space with base point $x_0$.
We recall from \cite{em} the definition of the first  $k$-invariant of~$X$.

 Let $p\colon [0,1]\to S^1=\{z\in \CC\, \vert\, \vert z\vert=1\}$
  be the map carrying $t\in [0,1]$ to $-i\, \exp(-2\pi i t)$.
  We provide $S^1$ with clockwise
 orientation and base point~$ -i=p(0)=p(1)$.
  Set $G=\pi_1(X,x_0)$  and recall that $A=\pi_2(X,x_0)$ is a left
$G$-module in the standard way.
  For  each $\alpha\in G$, fix a loop $  {u}_{\alpha} \colon  S^1\to X$ carrying
   $-i$ to $x_0$ and representing~$\alpha$.
   We assume that the loop ${u}_{\varepsilon}$  representing the  neutral  element $\varepsilon\in G$
is the constant path at $x_0$.
 Fix a 2-simplex $\Delta_2$ with vertices $v_0,v_1, v_2$.
 For every $\alpha, \beta\in G$,   fix a map
 ${f}_{\alpha, \beta}:\Delta_2 \to  X$
such that
  for all $t\in [0,1]$, $$ {f}_{\alpha, \beta} ((1-t)
v_0+t v_1)=  {u}_{\alpha}\, p(t),\, {f}_{\alpha, \beta} ((1-t) v_1+t v_2)=  {u}_{\beta}\, p(t)
,\, $$
$${f}_{\alpha, \beta} ((1-t) v_0+t v_2)=  {u}_{\alpha\beta}\, p(t) \, .
$$
We  assume that ${f}_{ \varepsilon, \varepsilon}$ is the constant map   $\Delta_2 \to \{x_0\}$
  and that for all $\alpha\in G$, the maps ${f}_{ \varepsilon, \alpha}$ and ${f}_{\alpha,  \varepsilon}$
   are the standard
\lq\lq constant" homotopies between two copies of~${u}_{\alpha}$.
 We call   the collection   $\Omega=\{\{{u}_{\alpha}\}_{\alpha\in G}; \{
 {f}_{\alpha, \beta}\}_{\alpha,\beta \in G}\}$   a {\it  basic system of loops and triangles in  $X$}.

Figures \ref{fig:kinvd}--\ref{fig:trace} below  represent
$X$-surfaces using the following conventions.
  The vertices in each figure are labeled by integers; the vertex   labeled by an integer $i$ is denoted $v_i$.
  With every edge   in the figure we associate an
  element of $G$ called its label. For some edges, the labels are indicated by   Greek letters in the picture.
  The labels of all the other edges  can be computed uniquely using
   the following rule:
 in
   any triangle   $v_i  v_j  v_k$ with $i<j<k$ the
 label  of  the edge $ v_iv_k $ is   the product of the labels of    $  v_i v_j $
   and $ v_j v_k $.   Each   figure below represents a compact oriented  surface
   $\Sigma$ endowed with a map $f: \Sigma \to X$.
 The map  $f$  carries  all vertices
 of $\Sigma$ to the base point $x_0\in X$.
   The restriction of $f$ to any  edge $ v_i v_j $ with $i<j$ and with label  $\rho\in G$
    is given by   $ f( (1-t)
v_i+t v_j )= u_\rho p (t)$ for all $t\in [0,1]$. The restriction of
$f$  to
 any triangle  $v_i  v_j  v_k$ with $i<j<k$  is given by   $$ f(t_0
v_i+t_1 v_j +t_2 v_k )= f_{\rho, \eta}  (t_0 v_0+t_1 v_1 +t_2
v_2),$$ where $\rho, \eta$ are the labels of $v_iv_j$ and $v_jv_k$
respectively and
  $t_0, t_1, t_2$ run over   non-negative real numbers whose sum is equal to
  $1$. The map $f:\Sigma\to X$ is considered up to
homotopy    constant on $\partial \Sigma$.

%%%We can view $f_{\alpha, \beta, \gamma}$  as a \lq\lq difference"
%%% between the two triangulations of the triangle shown in Figure \ref{fig:kinvdef}.
\begin{figure}[h, t]
    \centering
        \includegraphics[angle=0,width=0.4\textwidth]{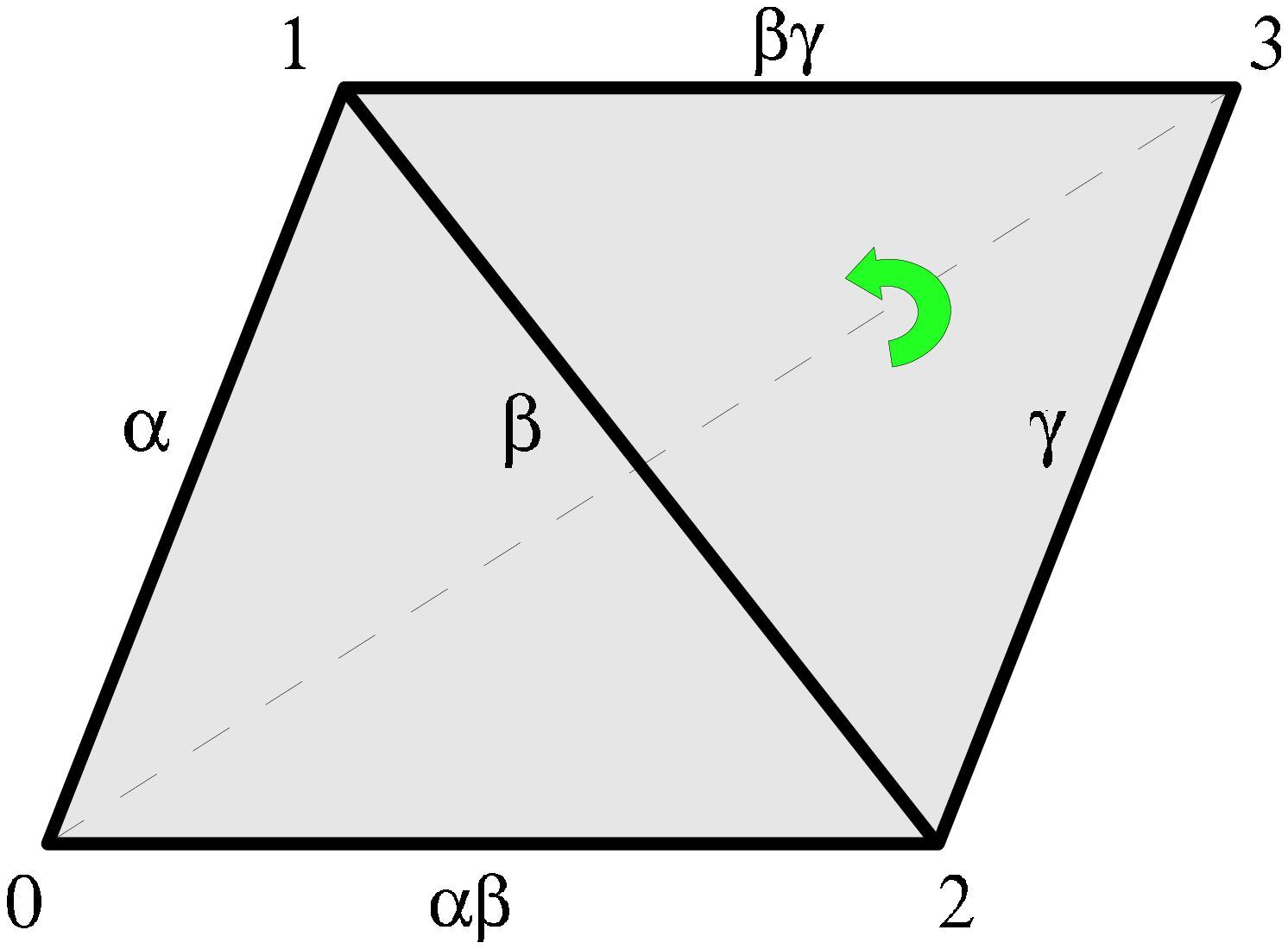}
    \caption{A map $ \partial \Delta_3 \to X $}
    \label{fig:kinvd}
\end{figure}

Let $\Delta_3$ be the standard 3-simplex   with vertices $v_0,v_1,v_2,v_3$.
 Given $\alpha, \beta, \gamma \in G$,   Figure \ref{fig:kinvd} describes
  a map   from the 2-sphere $\partial \Delta_3$ to $X$. Here
   the labels of  the edges $ v_0v_2 , v_1 v_3 $, and $v_0v_3$
  are $\alpha \beta, \beta \gamma$, and $\alpha \beta  \gamma$, respectively.
We take   $v_0$ as the base point of   $\partial \Delta_3$. With
this choice, the map $ \partial \Delta_3 \to X $ in Figure
\ref{fig:kinvd} represents an element of $A=\pi_2(X,x_0)$ denoted
$\kappa(\alpha, \beta, \gamma)$. This   defines a map
$\kappa=\kappa^\Omega: G^3\to A$.   It is known that $\kappa$ is a
cocycle. It follows from the definitions  (and from the   choice of
the maps ${f}_{ \alpha,\varepsilon},
  {f}_{  \varepsilon,\alpha }$)  that   $\kappa$ is normalized in the sense of Section \ref{preli}. The cohomology
class of $\kappa$ in $H^3(G, A)$ is independent of
 the choice of $\Omega$. This cohomology class is  the first $k$-invariant of
  $X$.

\section{The TF-algebra underlying a 2-dimensional  HQFT} $\,$

 Let   $X$ be a   path connected topological space with base point
 $x_0$.
 We   derive from any $2$-dimensional $X$-HQFT $({V},\tau)$      its {\it underlying TF-algebra}.

 For shortness, one-dimensional $X$-manifolds will be called {\it $X$-curves} and   two-dimensional $X$-cobordisms
 will be called {\it $X$-surfaces}. We say that two $X$-surfaces with the  same bases $$(W,M_0,M_1,g:W\to X) \quad {\text {and}}\quad
 (W',M_0,M_1,g':W'\to X)$$
 are {\it h-equivalent} if  there is a homeomorphism $F:W\to W'$ constant on the boundary
and such that
  $  g'F:W\to X $ is homotopic to $g$ via a homotopy  constant on the
 boundary. The axioms of an HQFT imply that   then
 $\tau(W)=\tau(W'):V_{M_0}\to V_{M_1}$.

Fix    a    basic system of loops and triangles
$\Omega=\{\{{u}_{\alpha}\}_{\alpha\in G}; \{
 {f}_{\alpha, \beta}\}_{\alpha,\beta \in G}\}$    in  $X$, where
 $G=\pi_1(X,x_0)$. Let  $\kappa:G^3\to A=\pi_2(X,x_0)$  be
 the cocycle constructed in Section 2.2.
 We construct a TF-algebra   $ {V}^\Omega$  over $(G, A, \kappa)$   in several steps.

Step 1 ({\it Underlying module and action of $A$}).  For each
$\alpha \in G$, consider the pointed oriented circle $S^1$
 and the map ${u}_\alpha: S^1\to X$ as in
Section \ref{XXX}.
The pair $(S^1, {u}_\alpha)$ is an $X$-curve. Set
 $V_{\alpha}={V}_{(S^1, {u}_\alpha)}$.
By the definition of an HQFT, $V_{\alpha}$ is a finite-dimensional $K$-vector space.

For all $\alpha\in G$, the group $A=\pi_2(X,x_0)$ acts on $V_{\alpha}$   by
\begin{eqnarray*}
av=\tau({S} (\alpha,a))(v).
\end{eqnarray*}
where $a\in A$,  $v\in V_{\alpha}$, and ${S}(\alpha, a)$ is the
$X$-annulus obtained as a connected sum of the cylinder cobordism
over $(S^1, {u}_\alpha)$ and a map $S^2\to X$ representing $a$. The
$X$-annulus ${S}(\alpha, a)$ is shown in the left part of Figure
\ref{fig:multiv}. Here and below  we use external arrows to indicate
the edges glued to each other (keeping the surface oriented).
 %The bases of ${S}(\alpha, a)$ are images under the gluing of the horizontal sides of the central rectangle.

  Observe that the $X$-annulus   in the right part of Figure
\ref{fig:multiv} is h-equivalent to  ${S}(\alpha, a)$. Therefore,
  $av=\,^{\alpha}av$ for all   $a\in A$ and  $v\in V_{\alpha}$.

\begin{figure}[h,t]
    \centering
        \includegraphics[width=0.6\textwidth]{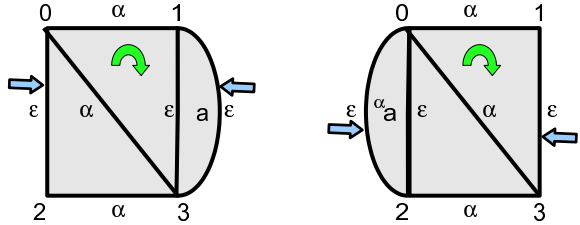}
    \caption{The $X$-annulus ${S}(\alpha, a)$}
    \label{fig:multiv}
\end{figure}

Step 2 ({\it Multiplication}). For   $\alpha, \beta\in G$, consider
the $X$-co\-bor\-dism (a disk with two holes) ${{D}}(\alpha, \beta)$
in Figure \ref{fig:prod}. Composing the
  map $$V_{\alpha} \times V_\beta \to V_{\alpha}\otimes
V_{\beta}, \, (u,v)\mapsto u\otimes v$$ with the homomorphism
  $$ \tau({{D}}(\alpha, \beta)):V_{\alpha}\otimes
V_{\beta}\to V_{\alpha\beta} $$  we obtain  a $K$-linear
multiplication $V_{\alpha} \times V_\beta \to V_{\alpha \beta}$.
This extends to multiplication in $ \oplus_{\alpha\in G}\, V_\alpha$
by linearity.

\begin{figure}[h,t]
    \centering
        \includegraphics[width=0.60\textwidth]{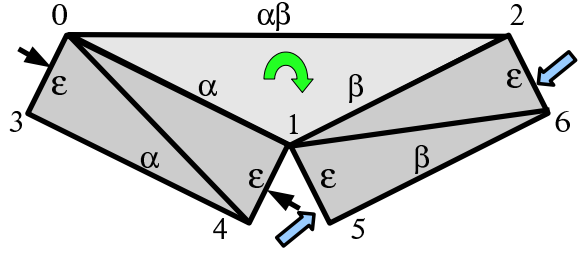}
    \caption{The $X$-disk with 2 holes ${{D}}(\alpha, \beta)$}
    \label{fig:prod}
\end{figure}

It is clear that for any $a\in A$, the glueing of ${S}(\alpha,a)$ to
${{D}}(\alpha, \beta)$  along the left bottom base of ${{D}}(\alpha,
\beta)$ and the  glueing of ${S}(\alpha\beta,a)$ to ${{D}}(\alpha,
\beta)$ along the top base of ${{D}}(\alpha, \beta)$ yield
h-equivalent $X$-surfaces. Applying~$\tau$, we obtain
 $(au)v=a(uv)$ for any $u\in V_\alpha$ and $v\in V_\beta$.
Similarly,  $(au)v=u(av)$. Hence, multiplication in
$\oplus_{\alpha\in G}\, V_\alpha $ is $K[A]$-bilinear.

To check the twisted associativity (\ref{associativity}),   note
that $(uv)w $ is computed by applying $\tau$ to the $X$-surface (a
disk with three holes)  obtained by glueing ${{D}}(\alpha, \beta)$
to ${{D}}(\alpha\beta, \gamma)$. Similarly, $u(vw)$ is computed by
applying $\tau$ to the $X$-surface obtained by glueing ${{D}}(\beta,
\gamma)$ to ${{D}}(\alpha,\beta \gamma)$. These $X$-surfaces are
shown in Figure $\ref{fig:asoci}$  where the external arrows
indicating the glueing of sides are omitted. It is easy to see that
the left $X$-surface is h-equivalent to a connected sum of the right
$X$-surface with the map $S^2\to X$ used to define
$\kappa(\alpha,\beta,\gamma) \in A$. This  implies
Formula~(\ref{associativity}). Note for future use that the narrow
\lq\lq collar-type" rectangles in Figure~\ref{fig:asoci} play no
role in the argument but are necessary to define the $X$-surfaces at
hand. In some of the figures below we  omit these collar rectangles
to simplify the figures.

%Evaluating $\tau$ on the cobordism from  Figure $\ref{fig:asoci}$ we get  that for  $u\in V_{\alpha}$, $v\in V_{\beta}$ and $w\in V_{\gamma}$:
%\begin{eqnarray*}
%(uv)w=(\kappa(\alpha,\beta,\gamma)u)(vw).
%\label{uvw}
%\end{eqnarray*}

\begin{figure}[h, t]
    \centering
        \includegraphics[width=0.8\textwidth]{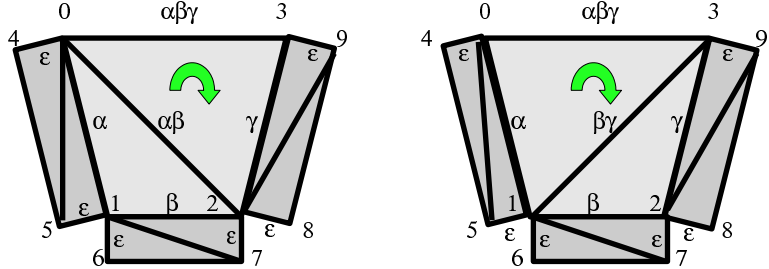}
    \caption{Proof of the identity $(uv)w= \kappa(\alpha,\beta,\gamma)\, u (vw)$}
    \label{fig:asoci}
\end{figure}
Next consider an oriented 2-disk ${{B}}$   mapped   to
$\{x_0\}\subset X$ and viewed as an $X$-cobordism between an empty
set and $(S^1, {u}_{\varepsilon})$. We have a map $\tau({{B}}):K\to
V_{\varepsilon}$, and we  set
 $1_V=\tau({{B}})(1) $.
One   easily sees that $1_Vu=u1_V=u$ for all $u\in V$.

Step 3 ({\it Inner product}). Consider the $X$-annulus
${{C_{--}}}(\alpha,\varepsilon)$ shown in  Figure~\ref{fig:nedf}. The subscript $--$ reflects the fact that the orientation on both boundary components of the annulus
is opposite to the orientation induced from  the annulus. Set
$$\eta_{\alpha}=\tau({{C_{--}}}(\alpha,\varepsilon)): V_{\alpha} \otimes V_{\alpha^{-1}}\to K.$$

\begin{figure}[h, t]
    \centering
        \includegraphics[angle=0,width=0.45\textwidth]{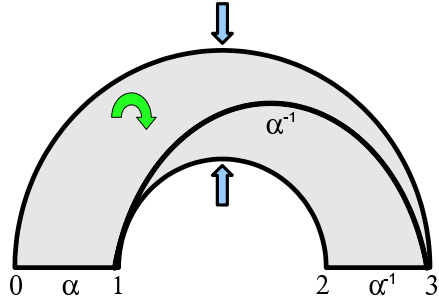}
    \caption{The $X$-annulus ${{C_{--}}}(\alpha,\varepsilon)$}
    \label{fig:nedf}
\end{figure}

 The $X$-annulus
${{C_{--}}}(\alpha,\varepsilon)$ can be obtained by glueing the
$X$-surfaces ${{C_{--}}}(\varepsilon,\varepsilon)$, ${{D}}(\alpha,
\alpha^{-1})$, and ${{B}}$. This yields
$$\eta_{\alpha}(u\otimes v)=\tau({{C_{--}}}(\alpha,\varepsilon))(u\otimes v)=\eta_{\varepsilon}(uv\otimes 1_V)$$
for all $u\in V_{\alpha}$ and $v\in  V_{\alpha^{-1}}$.

The   $X$-surfaces shown in Figure \ref{fig:pfned}  are
h-equivalent. Applying $\tau$ to the surface on the left we get
 $({\rm {id}}_{V_{\alpha^{-1}}}\otimes \eta_{\alpha})(\iota \otimes {\rm
 {id}}_{V_{\alpha^{-1}}})$, where $\iota$ is a   homomorphism
 $K\to {V_{\alpha^{-1}}}\otimes V_{\alpha}$.
Applying $\tau$ to the surface on the right we get ${\rm
{id}}_{V_{\alpha^{-1}}}$.
 The resulting equality $({\rm {id}}_{V_{\alpha^{-1}}}\otimes \eta_{\alpha})(\iota \otimes {\rm
 {id}}_{V_{\alpha^{-1}}})={\rm
{id}}_{V_{\alpha^{-1}}}$  implies the non-degeneracy of
$\eta_{\alpha}$.

 \begin{figure}[h, t]
    \centering
        \includegraphics[angle=0,width=0.7\textwidth]{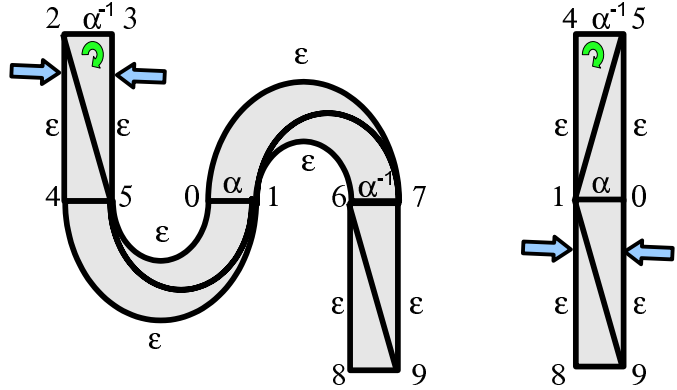}
    \caption{Proof of the non-degeneracy of
$\eta_{\alpha}$}
\label{fig:pfned}
\end{figure}

 Note finally that the map ${{C_{--}}}(\varepsilon,\varepsilon) \to X$ used to define  $\eta_{\varepsilon}$ is the constant map with value $x_0$. This easily implies that the form $\eta_{\varepsilon}$ is symmetric.

%{\bf If the space $X$ is acyclic then ${{C_{--}}}(\alpha,\varepsilon)$ is the same as the ${{C_{--}}}(\alpha,\varepsilon)$ described in \cite{tu} and \cite{tu2}. However one should notice that in the general case there is no canonical way of choosing one ${{C_{--}}}(\alpha,\varepsilon)$. This is way we  describe ${{C_{--}}}(\alpha,\varepsilon)$ in terms of the fixed basic system of loops and triangles. A similar statement is true for the annuli ${{C_{-+}}}(\alpha,\beta)$ used in the next step.}

Step 4 ({\it Projective action of G}). Consider the $X$-annulus
$C_{-+}(\alpha,\beta)$ shown  in  Figure~\ref{fig:phib}. The subscript $-+$ reflects the fact that the orientation on the boundary component corresponding to $\alpha$ is  opposite to the orientation induced from the annulus, while  the orientation on the  boundary component corresponding to $\beta\alpha\beta^{-1}$ is induced from the orientation of the annulus. Set
$$\varphi_{\beta}=\tau(C_{-+}(\alpha,\beta)):V_{\alpha}\to V_{\beta\alpha\beta^{-1}}.$$

\begin{figure}[h, t]
    \centering
        \includegraphics[angle=0,width=0.3\textwidth]{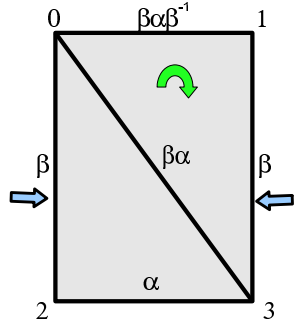}
    \caption{The $X$-annulus $C_{-+}(\alpha,\beta)$}
    \label{fig:phib}
\end{figure}

 The identity  $ \varphi_{\beta}(av)=\,^{\beta}a\, \varphi_{\beta}(v)$
for $a\in A$ and $v\in V_\alpha$ follows from the fact the
$X$-annulus obtained by glueing $C_{-+}(\alpha,\beta)$ to
${S}(\alpha,a)$ is h-equivalent to the $X$-annulus obtained by
glueing ${S}(\beta\alpha\beta^{-1},\,^{\beta}a)$ to
$C_{-+}(\alpha,\beta)$.

Figure \ref{fig:triv} represents two h-equivalent $X$-annuli (it is
understood that the vertical sides of the right rectangle are glued
in the usual way to make an annulus). The left $X$-annulus is
$C_{-+}(\alpha,\varepsilon)$. The right
 $X$-annulus is obtained by glueing $C_{-+}(\alpha,\beta)$
represented by the lower rectangle  to an $X$-annulus
$C'(\alpha,\beta)$ represented by the upper rectangle.   This proves
that $\varphi_{\beta}$ is invertible and
$\varphi_{\beta}^{-1}=\tau(C'(\alpha,\beta))$.

\begin{figure}[h, t]
    \centering
        \includegraphics[angle=0,width=0.7\textwidth]{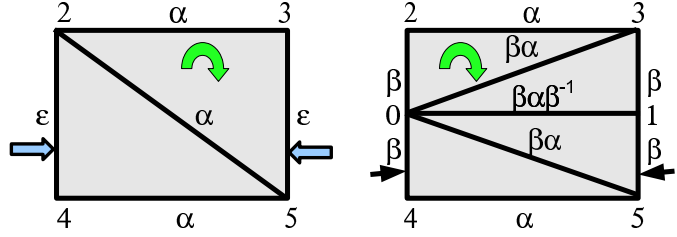}
    \caption{Proof of the invertibility of $\varphi_{\beta}$}
    \label{fig:triv}
\end{figure}

To prove that  $\varphi_{\beta} \vert_{V_{\beta}} ={\text
{id}}_{V_{\beta}}:V_{\beta} \to V_{\beta}$,  consider the $X$-annuli
in Figure~\ref{fig:dehn}. Using the Dehn twist of an annulus about
its core circle, one easily observes that   these two $X$-annuli are
h-equivalent. The $X$-annulus on the right is
$C_{-+}(\beta,\varepsilon)$, and the associated map
$\varphi_{\varepsilon}:V_{\beta}\to V_{\beta}$
  is the identity by Axiom (7) of an HQFT. It is
easy to see that a connected sum of the left $X$-annulus with the
map $S^2\to X$ used to define $\kappa(\beta, \varepsilon ,\beta) \in
A$  is h-equivalent to $C_{-+}(\beta,\beta)$. Since $\kappa(\beta,
\varepsilon ,\beta)=1$, we obtain  ${\rm
 {id}}_{V_{\beta}}=   \varphi_{\beta}
 \vert_{V_{\beta}} $.

\begin{figure}[h, t]
    \centering
        \includegraphics[width=0.70\textwidth]{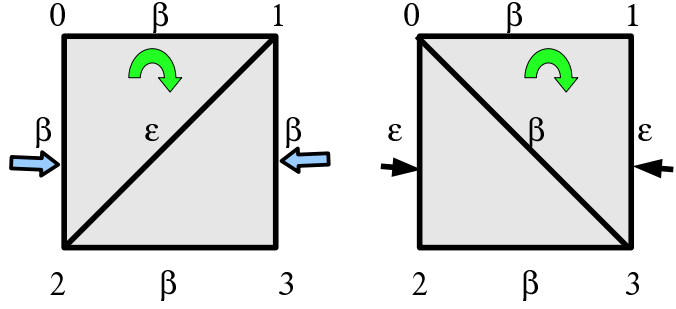}
    \caption{Proof of the identity $\varphi_{\beta} \vert_{V_{\beta}} ={\text
{id}}_{V_{\beta}}$}
    \label{fig:dehn}
\end{figure}

The equality $\varphi_\beta (1_V)=1_V$  follows from the fact that
the $X$-disk  obtained by glueing the $X$-disk ${{B}}$ to the bottom
of $C_{-+}(\varepsilon,\beta)$
 is h-equivalent to ${{B}}$.

To prove the identity \eqref{com}, consider   the
  $X$-surfaces (disks with two holes) in Figure \ref{fig:bapba} (we
   omit  in the figure the collar rectangles glued to   some  external edges; they play no role in the argument and
    the reader may easily recover them).
  The first
$X$-surface is obtained by glueing  the $X$-annulus $C'
(\alpha,\beta)$ to the right bottom boundary component $(S^1,
{u}_\alpha)$ of ${{D}}(\beta, \alpha)$. Since
$\kappa(\varepsilon,\beta,\alpha)=1$, this $X$-surface is
h-equivalent to the second $X$-surface in Figure \ref{fig:bapba}.
The third $X$-surface is obtained from the second one by separating
the  left face along the $\varepsilon$-labeled edge $v_0v_2$ and
gluing it on the right along the edges  indicated by the external
arrows. Thus, the second and third $X$-surfaces are h-homeomorphic.
The map from the third $X$-surface to $X$ can be easily computed
because its restriction to the face $v_1v_3v_4$ (respectively to
$v_0v_2v_4$) is the  constant homotopy   of $\beta$ to itself
  (respectively, of $\beta\alpha$ to itself).
This $X$-surface is h-equivalent to ${{D}}(\beta\alpha\beta^{-1},
\beta)$. Applying $\tau$, we obtain  $ v \varphi_{\beta}^{-1}(w)=wv$
for all $v\in V_{\beta}$ and $w\in V_{\beta\alpha\beta^{-1}}$. This
is an equivalent form of \eqref{com}.

\begin{figure}[h, t]
    \centering
        \includegraphics[angle=0,width=0.6\textwidth]{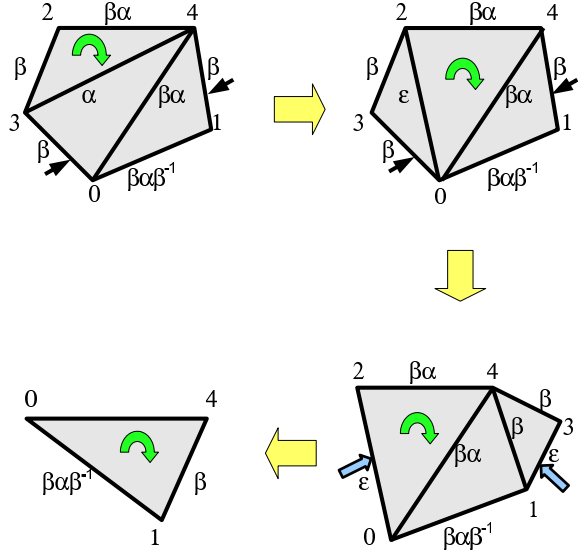}
    \caption{Proof of the identity $v\varphi_{\beta}^{-1}(w)=wv$}
    \label{fig:bapba}
\end{figure}

The identity $\eta_{\varepsilon}(\varphi_{\beta}(u) \otimes
\varphi_{\beta}(v))=\eta_{\varepsilon}(u\otimes v)$ for   $u,v\in
V_\varepsilon$ follows from the fact that the $X$-annulus obtained
by the glueing of $C_{-+}(\varepsilon,\beta)\amalg
C_{-+}(\varepsilon,\beta)$ to the bottom of
${{C_{--}}}(\varepsilon,\varepsilon)$  is h-equivalent to
${{C_{--}}}(\varepsilon,\varepsilon)$   (see Figure
\ref{fig:dualphig}).

\begin{figure}[h,t]
    \centering
        \includegraphics[angle=0,width=0.4\textwidth]{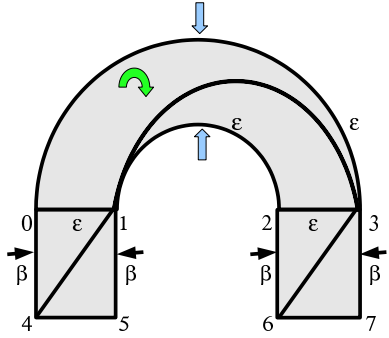}
    \caption{Proof of the identity $\eta_{\varepsilon}(\varphi_{\beta}(u) \otimes \varphi_{\beta}(v))=\eta_{\varepsilon}(u \otimes v)$}
    \label{fig:dualphig}
\end{figure}

To prove that $$\varphi_{\beta}(u)\, \varphi_{\beta}(v) =l_{\alpha,\gamma}^{\beta} \, \varphi_{\beta}(uv)$$
 for all $u\in V_{\alpha}$ and $v\in V_{\gamma}$, consider the $X$-disks with three holes $W_1, W_2, W_3, W_4$
  in   Figure \ref{fig:proab}
 (again we omit the collar rectangles glued to some external edges).
Clearly, the homomorphism $\tau(W_1): V_{\alpha} \otimes
V_{\gamma}\to V_{\beta \alpha\gamma
 \beta^{-1}}$ carries
 $  u\otimes v$ to $\varphi_{\beta}(u)\,
 \varphi_{\beta}(v)$.
The   $X$-surface $W_2$ differs from $W_1$ by   two adjacent
triangles $v_0v_{1'} v_5$ and $v_0v_{1''} v_5$ both representing
$f_{\beta \alpha  \beta^{-1}, \beta \gamma \beta^{-1}}:\Delta_2\to
X$. These two copies of  $f_{\beta \alpha  \beta^{-1}, \beta \gamma
\beta^{-1}}$  \lq\lq cancel" each other, and so  $W_2$ is
h-equivalent to~$W_1$. Thus,  $\tau(W_2)=\tau(W_1)$. The $X$-surface
$W_3$ is obtained from $W_2$ by removing the vertices $v_{1'} $ and
$ v_{1''}$. Hence
$$\tau(W_3)=  \kappa(\beta\alpha\beta^{-1},\beta\gamma\beta^{-1},\beta)^{-1} \,\kappa(\beta\alpha\beta^{-1},\beta,\gamma)\, \tau(W_2).$$
The $X$-surface $W_4$ is obtained from $W_3$ by switching the
diagonal in the quadrilateral $v_0v_3v_4v_5$. Therefore
$\tau(W_4)=\kappa(\beta, \alpha,\gamma)^{-1}  \, \tau(W_3)$.
Combining these formulas, we obtain
 $\tau(W_4) =(l_{\alpha,\gamma}^{\beta})^{-1} \, \tau(W_1) $. It remains to observe that
the homomorphism $\tau(W_4): V_{\alpha} \otimes  V_{\gamma}\to
V_{\beta \alpha\gamma
 \beta^{-1}}$ carries
 $  u\otimes v$ to $\varphi_{\beta}(uv)$.

\begin{figure}[h,t]
    \centering
        \includegraphics[width=0.7\textwidth]{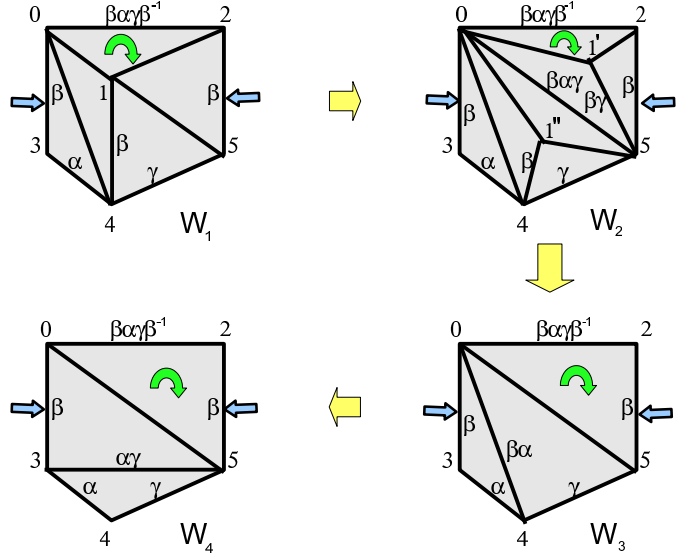}  \caption{Proof of the identity $\varphi_{\beta}(u)\, \varphi_{\beta}(v) =l_{\alpha,\gamma}^{\beta} \, \varphi_{\beta}(uv)$}
    \label{fig:proab}
\end{figure}

To show the identity $\varphi_{\gamma\beta}\vert_{V_\alpha}
=h_{\gamma,\beta}^{\alpha}\, \varphi_{\gamma}
\varphi_{\beta}\vert_{V_\alpha}$, consider the $X$-annuli $W_1$,
$W_2$, $W_3$, $ W_4$ in  Figure \ref{fig:phigb}. Clearly,
$W_1=C_{-+}(\alpha,\gamma\beta)$ and therefore $\tau(W_1)=
\varphi_{\gamma\beta} \vert_{V_\alpha}$. The $X$-annulus $W_2$ is
obtained from $W_1$ by adding two vertices $v_{2} $ and $ v_{3}$.
Hence
$$\tau(W_2)=  \kappa(\gamma\beta\alpha\beta^{-1}\gamma^{-1},\gamma,\beta)^{-1}\, \kappa(\gamma,\beta, \alpha)^{-1}\, \tau(W_1).$$
The $X$-annulus $W_3$ is obtained from $W_2$ by switching the
diagonal in the quadrilateral $v_0v_2v_5v_3$. Therefore $\tau(W_3)=
 \kappa(\gamma,\beta\alpha\beta^{-1},\beta) \,\tau(W_2)$. Finally,
 $W_4$ is obtained from $W_3$ by canceling two adjacent copies of the
 singular triangle $f_{\gamma, \beta}$. Therefore $W_4$ is
  h-equivalent to $W_3$ and $\tau(W_4)=\tau(W_3)$. It is clear  that $\tau(W_4)= \varphi_{\gamma} \varphi_{\beta} \vert_{V_\alpha}$.
   Combining these
  equalities, we obtain the required identity.

\begin{figure}[h,t]
    \centering
        \includegraphics[angle=0,width=0.7\textwidth]{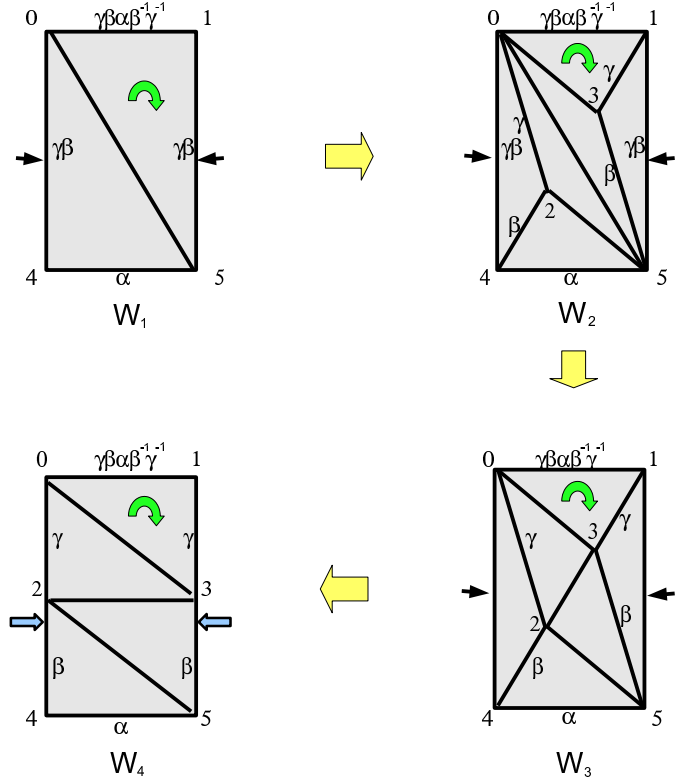}
    \caption{Proof of the identity $\varphi_{\gamma\beta}(u)= h_{\gamma,\beta}^{\alpha}\, \varphi_{\gamma}(\varphi_{\beta}(u))$}
    \label{fig:phigb}
\end{figure}

Step 5 ({\it The trace condition}). Consider the three  $X$-surfaces
(punctured tori) $W_1, W_2, W_3$   in Figure \ref{fig:trace}. All
three are $X$-cobordisms between $(S^1, u_{\gamma})$ and
$\emptyset$, where $\gamma= \alpha\beta\alpha^{-1}\beta^{-1}$.
Clearly,
$$\tau(W_1)=\kappa(\gamma,\beta\alpha\beta^{-1},\beta)\,
\tau(W_2) \quad {\text {and}} \quad
\tau(W_3)=\kappa(\gamma,\beta,\alpha)\, \tau(W_2).$$
Consider the $X$-surface obtained by glueing  $C_{-+}(\alpha,\beta)$ and  ${{D}}(\alpha,\gamma)$. If we identify the two copies of $(S^1, {u}_\alpha)$ we obtain $W_1$. Similarly for the $X$-surface obtained by glueing ${{D}}(
\gamma,\beta)$ and $C'(\beta,\alpha)$ we can  identify the two copies of $(S^1, {u}_{\beta})$ in order to obtain $W_3$. This implies the equality
$${\rm {Tr}}(\kappa(\gamma,\beta,\alpha)\mu_c\varphi_{\beta}:V_{\alpha}\to V_{\alpha})= {\rm {Tr}}(\kappa(\gamma,\beta\alpha\beta^{-1},\beta)
\varphi_{\alpha}^{-1}\mu_c:V_{\beta}\to V_{\beta}).
$$
\begin{figure}[h, t]
    \centering
        \includegraphics[angle=0,width=0.7\textwidth]{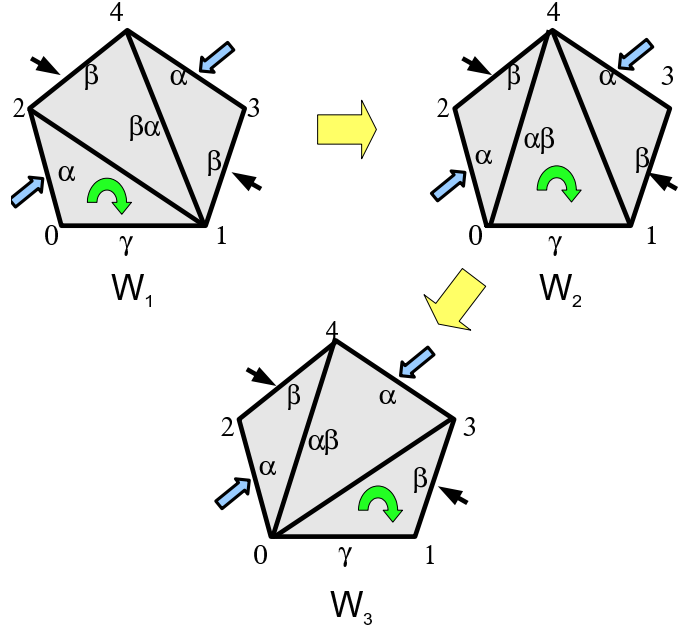}
    \caption{Proof of the trace condition}
    \label{fig:trace}
\end{figure}

We summarize the results above in the following theorem.

\begin{theorem}
To every 2-dimensional $X$-HQFT  $(V,\tau)$   and to every basic
system of loops and triangles $\Omega=(\{{u}_{\alpha}\}_{\alpha \in
G},\{{f}_{\alpha, \beta}\}_{\alpha, \beta\in G})$ in $X$, we
associate a TF-algebra $V^{\Omega}=V^{\Omega}(\tau)$ over the triple
$(G=\pi_1(X),A=\pi_2(X),\kappa^\Omega:G^3\to A)$.
\end{theorem}

It is obvious that the   construction of $V^{\Omega} $ is functorial
with respect to isomorphisms of $X$-HQFTs. We now show that
$V^{\Omega} $ does not depend on $\Omega$ at least up to coboundary
equivalence (see Section~1.4).

\begin{lemma}   The   TF-algebra $V^{\Omega} $ considered  up to coboundary
equivalence
does not depend on the choice of
 $\Omega$.\label{omega}
\end{lemma}
\begin{proof}
Let $\Omega'=(\{u'_{\alpha}\}_{\alpha \in G},\{{f}'_{\alpha,
\beta}\}_{\alpha, \beta\in G})$ be another basic system of loops and
triangles in~$X$. Assume first that $u_{\alpha}=u'_{\alpha}$ for all
$\alpha\in G$. It is clear  that up to homotopy constant
on~$\partial \Delta^2$  the map ${f}_{\alpha, \beta}: \Delta^2 \to
X$ is a connected sum of  ${f}'_{\alpha, \beta}: \Delta^2 \to X$ and
a certain map $\omega(\alpha,\beta):    \Delta^2 \to X$ carrying
$\partial \Delta^2$ to $x_0$. Assigning to each pair $(\alpha,
\beta)\in G^2$  the element of $A=\pi_2(X,x_0)$ represented by
$\omega(\alpha,\beta)$, we obtain a 2-cochain $\omega:G^2\to A$.
This cochain is normalized because by the definition of a basic
system of loops and triangles,   $f'_{\alpha, \varepsilon}=
f_{\alpha, \varepsilon}$ and $f'_{ \varepsilon, \alpha}= f_{
\varepsilon, \alpha}$ for all $\alpha\in G$.    A direct comparison
of the
 $X$-surfaces used to define the 3-cocycles
$\kappa,\kappa':G^3\to A$ associated with $\Omega, \Omega'$ and the
TF-algebras $V^{\Omega} , V^{\Omega'} $ shows that
$\kappa'=\delta^2(\omega)\kappa$ and
  $V^{\Omega'} $
is obtained from $V^{\Omega} $ by the coboundary transformation
determined by $\omega$, i.e.,
 $V^{\Omega'} \cong (V^{\Omega} )^{\omega}$.

In the   case   $u_{\alpha}\neq u'_{\alpha}$ for some $\alpha$, we
construct a third basic system of loops and triangles   in~$X$ as
follows.
 For each $\alpha\in G$,  fix a homotopy $h_\alpha:S^1\times [0,1]\to X$ between the loops
 $u_{\alpha}, u'_{\alpha}:S^1 \to X$ representing $\alpha$. (It is understood that $h_\varepsilon$ is a constant map to $x_0$.) Recall the map $p:[0,1]\to S^1$ from Section~2.2.
 The
 map $$\overline h_\alpha=h_\alpha \circ (p\times {\rm {id}}_{[0,1]}): [0,1]^2\to X$$ is a homotopy between
  $u_\alpha p$ and $u'_\alpha
 p$.
 We construct a map
${f}^h_{\alpha, \beta} :\Delta^2\to X$ by gluing   $ \overline
h_\alpha, \overline h_\beta, \overline h_{\alpha\beta}$ to the sides
of the singular simplex ${f}_{\alpha, \beta}:\Delta^2\to X$. The
system $\Omega^h= (\{u'_{\alpha}\}_{\alpha\in  G},\{{f}^h_{\alpha,
\beta}\}_{\alpha, \beta\in  G})$ satisfies all conditions of  a
basic system of loops and triangles in~$X$ except possibly one: the
maps ${f}^h_{ \varepsilon, \alpha}$ and ${f}^h_{\alpha,
\varepsilon}$
   are not necessarily the standard
\lq\lq constant" homotopies between two copies of~$u'_{\alpha}p$.
However,   ${f}^h_{ \varepsilon, \alpha}$ and ${f}^h_{\alpha,
\varepsilon}$  are homotopic rel~$\partial \Delta^2$ to these
constant homotopies, as easily follows from the assumptions on
${f}_{ \varepsilon, \alpha}$ and ${f}_{\alpha,  \varepsilon}$ and
the definition of $f^h$. Therefore, deforming if necessary the maps
${f}^h_{ \varepsilon, \alpha}$ and ${f}^h_{\alpha, \varepsilon}$, we
can transform   $ \Omega^h$ into a basic system $\Omega''$ of loops
and triangles in~$X$.   The associated  3-cocycle $\kappa'':G^3\to
A$ is equal to $\kappa$.
  Moreover, the   homomorphisms
$$\{\tau(h_\alpha):V_{u_{\alpha}}\to V_{u'_{\alpha}}\}_{\alpha\in G}$$ form
 an isomorphism   $V^{\Omega} \to V^{\Omega''} $  in the category of  TF-algebras over $(G,A,\kappa)$.
 By the argument above, the TF-algebra $V^{\Omega'} $ is
 obtained from $V^{\Omega''} $ by a coboundary
 transformation. This completes the proof of the lemma.
\end{proof}

The following theorem shows that the isomorphism class of a
2-dimensional $X$-HQFT is entirely determined by the underlying
TF-algebra.

\begin{theorem}\label{polO}
Two 2-dimensional $X$-HQFTs are isomorphic if and only if their
underlying TF-algebras are coboundary equivalent.
\end{theorem}

\begin{proof}  We fix a basic system $\Omega$ of loops and triangles in $X$. It is obvious that an  isomorphism of  $X$-HQFTs $\rho:(V_1,\tau_1)\to (V_2,\tau_2)$ induces an isomorphism
 $V_1 ^{\Omega}\simeq V_2 ^{\Omega}$ of  TF-algebras over $(\pi_1(X),\pi_2(X),\kappa^{\Omega})$. Therefore, the underlying  TF-algebras are coboundary equivalent.

Conversely, assume that we have two $X$-HQFTs $(V_1,\tau_1)$ and $(V_2,\tau_2)$ such that their underlying TF-algebras are coboundary equivalent.  If we use the same basic system $\Omega=(\{{u}_{\alpha}\}_{\alpha \in
\pi_1(X)},\{{f}_{\alpha, \beta}\}_{\alpha, \beta\in \pi_1(X)})$ of loops and triangles in $X$, then the
two  algebras in question are TF-algebras over the same triple
$(\pi_1(X),\pi_2(X),\kappa^{\Omega})$. Since they  are coboundary equivalent, there is
an isomorphism $\rho:V_1^{\Omega}\to V_2 ^{\Omega}$. We lift $\rho$ to an isomorphism
of the HQFTs.  For every $\alpha\in \pi_1(X)$, denote by $\rho_{\alpha}$ the restriction of $\rho$ to  $(V_1^{\Omega})_{\alpha}=(V_1)_{(S^1,u_{\alpha})}$. For
every connected $X$-curve $(M,g_M)$, there are $\alpha\in \pi_1(X)$, an  $X$-annuli $a=a_M:[0,1]\times S^1 \to X$, and an $X$-homeomorphism $h=h_M:(S^1,a_M\vert_{1\times S^1})\to (M,g_M)$ such that
$$a_M \vert_{0\times S^1}= u_{\alpha}:S^1\to X \; \; {\rm and} \; \; a_M([0,1]\times \{*\})= x_0,$$
where $*$ is the base point of $S^1$ and $x_0$ is the base point of $X$. Notice that $\alpha$ is uniquely determined by $(M,g_M)$. We define $\rho_M:(V_1)_M\to (V_2)_M$ by the formula
$$\rho_M=h_{\#_2}\tau_2(a)\rho_{\alpha}\tau_1(a)^{-1}h_{\#_1}^{-1},$$
where $h_{\#_i}:(V_i)_{(S^1,a\vert_{1\times S^1})}\to (V_i)_{(M,g_M)}$ is the $K$-isomorphism associated to $h$ by the HQFT  $(V_i,\tau_i)$. Next we show that $\rho_M$ does not depend on the choice of $a$ and $h$. Indeed, take another pair $\hat{a}$, $\hat{h}$  which satisfies the above properties. Consider $a$ (respectively $\hat{a}$) as an $X$-cobordism between $(S^1,u_{\alpha})$ and $(S^1,a\vert_{1\times S^1})$ (respectively $(S^1,\hat{a}\vert_{1\times S^1})$) and observe that $\hat{h}^{-1}h$ is an $X$-homeomorphism between the top bases of these $X$-cobordisms. Let $C$ be the $X$-cobordism obtained by glueing these two along $\hat{h}^{-1}h$ (the orientation in the second cobordism should be reversed).
Then $C$ is a cobordism   between two copies of $(S^1,u_{\alpha})$.  Clearly, $C$ is the trivial $X$-annuli (as in axiom (7) of an HQFT) up to h-equivalence and the action of some 
$c\in \pi_2(X)$. Hence $\tau_i(C):(V_i^{\Omega})_{\alpha}\to (V_i^{\Omega})_{\alpha}$ is multiplication by $c$.  Since $\rho$ is an isomorphism of TF-algebras, we have
$$\rho_{\alpha}\tau_1(C)=\tau_2(C)\rho_{\alpha}$$
or equivalently
$$\rho_{\alpha}\tau_1(\hat{a})^{-1}(\hat{h}^{-1}h)_{\#_1}\tau_1(a)=\tau_2(\hat{a})^{-1}(\hat{h}^{-1}h)_{\#_2}\tau_2(a)\rho_{\alpha}.$$
This implies that
$$\hat{h}_{\#_2}\tau_2(\hat{a})\rho_{\alpha}\tau_1(\hat{a})^{-1}\hat{h}_{\#_1}^{-1}=h_{\#_2}\tau_2(a)\rho_{\alpha}\tau_1(a)^{-1}h_{\#_1}^{-1}.$$
So $\rho_M$ does not depend on the choice of $a$ and $h$. Denote by $\overline{\rho}$  the family of $K$-isomorphisms $\{\rho_M:(V_1)_M\to (V_2)_M\}_M$  where $M$ runs over all $X$-curves. We claim that $\overline{\rho}$ is
an isomorphism of HQFT's.

For any $X$-homeomorphism $f:M\to N$, we have
\begin{eqnarray*}
f_{\#_2}\rho_M&=&f_{\#_2}(h_M)_{\#_2}\tau_2(a_M)\rho_{\alpha}\tau_1(a_M)^{-1}(h_M)_{\#_1}^{-1}\\
&=&(f\circ h_M)_{\#_2}\tau_2(a_M)\rho_{\alpha}\tau_1(a_M)^{-1}(f\circ h_M)_{\#_1}^{-1}f_{\#_1}\\
&=&\rho_Nf_{\#_1}.
\end{eqnarray*}
So, $\overline{\rho}$ is natural with respect to $X$-homeomorphisms.

Next, notice that  any compact oriented $X$-surface splits along a finite family
of disjoint simple loops into a disjoint union of  $X$-disks with at most  two holes. Deforming if necessary the map from the surface to $X$, we can choose the splitting loops to be
$X$-curves $X$-homeomorphic to  $(S^1,u_{\alpha})$, for some $\alpha\in \pi_1(X)$. Moreover, we can make the cuts in such a way that each component adjacent to the boundary is $X$-homeomorphic to some $a_M$ and all the other $X$-surfaces are of types
\begin{eqnarray}
B,\; S(\alpha,a),\; D(\alpha,\beta),\; C_{-+}(\alpha,\beta),\; C_{--}(\varepsilon,\varepsilon),\; C_{++}(\varepsilon,\varepsilon).
\label{surf}
\end{eqnarray}
We consider the $X$-surface $a_M(id\times h_M^{-1}):[0,1]\times M\to X$ which is  an $X$-cobordism between $M'=(M,u_{\alpha}\circ h_M^{-1})$ and $(M,g_M)$. Axiom (4) of an HQFT implies that for $i=1,2,$
$$\tau_i(a_M(id\times h_M^{-1}))=(h_M)_{\#_i}\tau_i(a_M)(h_M)_{\#_i}^{-1}.$$
From this formula and the definition of $\overline{\rho}$, we have
\begin{eqnarray*}
\rho_M\tau_1(a_M(id\times h_M^{-1}))
&=&(h_M)_{\#_2}\tau_2(a_M)\rho_{\alpha}(h_M)_{\#_1}^{-1}\\
&=&(h_M)_{\#_2}\tau_2(a_M)(h_M)_{\#_2}^{-1}(h_M)_{\#_2}\rho_{\alpha}(h_M)_{\#_1}^{-1}\\
&=&\tau_2(a_M(id\times h_M^{-1}))\rho_{M'}.
\end{eqnarray*}
Therefore the natural square diagram associated to $\overline{\rho}$ and  $a_M(id\times h_M^{-1})$ is commutative. Since $\rho$ is a morphism of TF-algebras,  the natural square diagrams associated with the $X$-surfaces listed in  (\ref{surf}) are commutative. Finally, since any $X$-surface can be obtained by glueing a finite collection of the above $X$-surfaces along $X$-homeomorphisms, we get that the natural square diagrams associated to all $X$-surfaces are commutative. We conclude that $\overline{\rho}$ is an isomorphism of HQFT's.
\end{proof}

We expect that any TF-algebra can be realized by a 2-dimensional
HQFT  which if true, together with
Theorem~\ref{polO} would yield  a complete algebraic
characterization of 2-dimensional HQFTs with arbitrary targets.

  Given a   mapping of pointed topological spaces
 $f:Y\to X$, we can pull back
 a  2-dimensional $X$-HQFT along $f$ to obtain a  2-dimensional $Y$-HQFT. If $f$ induces  isomorphisms in $\pi_1$ and $\pi_2$,
  then   the underlying TF-algebra  of this
 $Y$-HQFT
is isomorphic to the underlying TF-algebra  of the original
 $X$-HQFT.

\section{Simple TF-algebras and     cohomological
HQFTs}\label{uTFC}

In  this section we compute the underlying TF-algebras of  the
   cohomological 2-dimensional HQFTs. We begin by introducing
  a class of  simple TF-algebras.

\subsection{Simple TF-algebras and $\kappa$-pairs} Let
 $G$ be a group with neutral element $\varepsilon$
and  $A$ be a left $G$-module.   Let $\kappa:G^3\to A$ be a
normalized 3-cocycle. A TF-algebra $V=\oplus_{\alpha \in G} \,
V_\alpha$ over $(G,A,\kappa)$ is {\it simple} if ${\text {dim}}_K
(V_{\alpha})=1$ for all $\alpha\in G$. We now classify simple
TF-algebras in terms of so-called $\kappa$-pairs.

Form now on we endow the multiplicative abelian group $K^*$ with the
trivial action of $G$. This allows us to apply notation of Section
\ref{preli} to $K^*$-valued  cochains on $G$. By a {\it
$\kappa$-pair}, we   mean a pair of maps $g_1:G\times G\to K^*$,
$g_2:A\to K^* $ such that $g_2$ is a $\mathbb{Z} [G]$-homomorphism,
\begin{eqnarray}\label{uuu}
 g_1(\alpha,\varepsilon)=g_1(\varepsilon, \alpha)=1 \quad {\text {for all}}\quad   \alpha \in G, \label{con3+}
\end{eqnarray}
and
\begin{eqnarray}
 \delta^2 (g_1) =g_2\circ \kappa:G^3 \to K^*.  \label{con3}
\end{eqnarray}
The $\mathbb{Z}[G]$-linearity of   $g_2:A\to K^*$ means
  that $g_2$ is a group homomorphism  such that for all
  $\alpha \in G$, $a\in A$,
\begin{eqnarray}\label{lol}
g_2(\,^\alpha a)=g_2(a).
\end{eqnarray}
For example,   for any map $\psi :G\to K^*$ the pair
$$(g_1=\delta^1(\psi):G\times G\to K^*, g_2=1:A\to K^*)$$ is a
$\kappa$-pair. We call it a {\it coboundary $\kappa$-pair}.

\begin{lemma} \label{example}
  Let  $(g_1, g_2)$ be a $\kappa$-pair. For each $\alpha\in G$,
  let $V_\alpha$ be the one-dimensional vector space over $K$ with basis vector $l_\alpha$.
  We provide $V=\oplus_{\alpha\in G} V_\alpha$ with a structure of an $A$-module by
  $a v= g_2(a) v$ for all $v\in V$.
   We provide $V $ with $K$-bilinear multiplication
by
\begin{eqnarray}
 l_{\alpha} l_{\beta} =g_1(\alpha,\beta)^{-1}l_{\alpha\beta}
\end{eqnarray}
for all $\alpha, \beta \in G$.   Let $\eta_\varepsilon
:V_\varepsilon\otimes V_\varepsilon\to K$ be the $K$-bilinear form
such that
\begin{eqnarray}
\eta_\varepsilon (l_\varepsilon\otimes l_\varepsilon)=g_1(\varepsilon,\varepsilon).
\end{eqnarray}
Let $\varphi_{\beta}:V \to V $ be the $K$-homomorphism defined by
\begin{eqnarray}
\varphi_{\beta}(l_{\alpha})=g_1(\beta,\alpha)^{-1}g_1(\beta\alpha\beta^{-1},\beta) \, l_{\beta\alpha\beta^{-1}}.
\end{eqnarray}
for all $\alpha, \beta \in G$. Then the $G$-graded vector space $V $
with this data is a TF-algebra over $(G,A,\kappa)$.
\end{lemma}

The TF-algebra constructed in this lemma   is denoted by
$V(g_1,g_2)$.

\begin{proof}  The $K[A]$-bilinearity of multiplication in  $V$
follows from the definitions. Formula (\ref{bimodule}) follows from
the identity \eqref{lol}. It is sufficient to verify
(\ref{associativity}) for the basis vectors $u=l_\alpha, v=l_\beta,
w= l_\gamma$:
\begin{eqnarray*}
%m_{\alpha\beta,\gamma}(m_{\alpha,\beta}\otimes id_{\gamma})(u\otimes v)\otimes w)\\
 (l_{\alpha}  l_{\beta})  l_{\gamma}
&=& (g_1(\alpha,\beta)^{-1}l_{\alpha\beta}) \,  l_{\gamma} \\
&=&g_1(\alpha,\beta)^{-1}g_1(\alpha\beta,\gamma)^{-1}l_{\alpha\beta\gamma}\\
&=&g_1(\beta,\gamma)^{-1}g_1(\alpha,\beta\gamma)^{-1}g_2(\kappa(\alpha,\beta,\gamma))\, l_{\alpha\beta\gamma}\\
&=&  g_2( \kappa(\alpha,\beta,\gamma)) \, l_{\alpha} (l_{\beta} l_{\gamma}) \\
&=&\kappa(\alpha,\beta,\gamma)  \,  l_{\alpha}  (l_{\beta} l_{\gamma}) .
\end{eqnarray*}
Formula \eqref{con3+} implies that $l_\varepsilon$ is the unit
element of $V$. The symmetry and non-degeneracy of
$\eta_\varepsilon$ are obvious. The non-degeneracy of the form
$\eta_\alpha:V_\alpha \otimes V_{\alpha^{-1}}\to K$ defined by
\eqref{nas} follows from the formula $\eta(l_{\alpha }\otimes
l_{\alpha ^{-1}}) = g_1(\alpha , \alpha^{-1})^{-1}$.
%%%\begin{eqnarray*}
%%%\eta(l_{\alpha \beta}\otimes l_{(\alpha\beta)^{-1}})
%%%&=&g_1(\alpha\beta,(\alpha\beta)^{-1})^{-1}g_1(\alpha,\beta)^{-1}\\
%%%&=&g_1(\beta,(\alpha\beta)^{-1})^{-1}g_1(\alpha,\alpha^{-1})^{-1}g_2(k(\alpha,\beta,(\alpha\beta)^{-1}))\\
%%%&=&\eta(\kappa(\alpha,\beta,(\alpha\beta)^{-1})l_{\alpha}\otimes
%%%l_{\beta}l_{(\alpha\beta)^{-1}}).
%%%\end{eqnarray*}
  We
check (\ref{mor}):
\begin{eqnarray*} \varphi_{\beta}(al_{\alpha})&=&g_1(\beta,\alpha)^{-1}g_1(\beta\alpha\beta^{-1},\beta)\, g_2(a) \, l_{\beta\alpha\beta^{-1}}\\
&=&g_2(\,^{\beta}a)\, g_1(\beta,\alpha)^{-1}g_1(\beta\alpha\beta^{-1},\beta)\, l_{\beta\alpha\beta^{-1}}\\
&=&\,^{\beta}a\, \varphi_{\beta}(l_{\alpha}).
\end{eqnarray*}
Formulas  (\ref{id}) and \eqref{mor+} follow from the definitions.
 We check (\ref{com}):
\begin{eqnarray*}
 \varphi_{\beta}(l_{\alpha})\,  l_{\beta} &=& g_1(\beta,\alpha)^{-1}g_1(\beta\alpha\beta^{-1},\beta
 ) \, l_{\beta\alpha\beta^{-1}}\,  l_{\beta} \\
&=&g_1(\beta,\alpha)^{-1} \, l_{\beta\alpha} = l_{\beta}\,  l_{\alpha}.
\end{eqnarray*}
 Similar computations prove  (\ref{epg}) and (\ref{pd}).
 We now check (\ref{pgb}). Observe that for $\alpha, \beta, \gamma\in G$,
$$
\varphi_{\gamma}(\varphi_{\beta}(l_{\alpha}))=g_1(\gamma,\beta\alpha\beta^{-1})^{-1}
g_1(\gamma\beta\alpha(\gamma\beta)^{-1},\gamma)\,
g_1(\beta,\alpha)^{-1}g_1(\beta\alpha\beta^{-1},\beta)\,l_{\gamma\beta\alpha(\gamma\beta)^{-1}}
$$
and
\begin{eqnarray*}
\varphi_{\gamma\beta}(l_{\alpha})&=&g_1(\gamma\beta,\alpha)^{-1}
g_1(\gamma\beta\alpha(\gamma\beta)^{-1},\gamma\beta)\, l_{\gamma\beta\alpha(\gamma\beta)^{-1}}.
\end{eqnarray*}
Therefore $\varphi_{\gamma\beta}(l_{\alpha})=h \,
\varphi_{\gamma}(\varphi_{\beta}(l_{\alpha}))$, where
$$h= g_1(\gamma\beta,\alpha)^{-1}
g_1(\gamma\beta\alpha(\gamma\beta)^{-1},\gamma\beta)\,  \times $$
$$\times \,  g_1(\gamma,\beta\alpha\beta^{-1})\,
g_1(\gamma\beta\alpha(\gamma\beta)^{-1},\gamma)^{-1}
g_1(\beta,\alpha)\, g_1(\beta\alpha\beta^{-1},\beta)^{-1}.$$ Now, a
direct computation using the definition of
$h_{\gamma,\beta}^{\alpha}$ and the assumption that $g_2$ is a group
homomorphism satisfying $ g_2\circ \kappa=\delta^2 (g_1)$ shows that
$ g_2(h_{\gamma,\beta}^{\alpha})=h$. Therefore
$$\varphi_{\gamma\beta}(l_{\alpha})=h \,
\varphi_{\gamma}(\varphi_{\beta}(l_{\alpha})) =h_{\gamma,\beta}^{\alpha} \,
\varphi_{\gamma}(\varphi_{\beta}(l_{\alpha})).$$

To check the trace identity (\ref{tr}), observe that
\begin{eqnarray*}
&&g_2(\kappa(\alpha\beta\alpha^{-1}\beta^{-1},\beta,\alpha))= (\delta^2 (g_1)) (\alpha\beta\alpha^{-1}\beta^{-1},\beta,\alpha)\\
&=&g_1(\beta,\alpha)\, g_1(\alpha\beta\alpha^{-1},\alpha)^{-1}g_1(\alpha\beta\alpha^{-1}\beta^{-1},\beta\alpha)\, g_1(\alpha\beta\alpha^{-1}\beta^{-1},\beta)^{-1}.
\end{eqnarray*}
and similarly
\begin{eqnarray*}
&&g_2(\kappa(\alpha\beta\alpha^{-1}\beta^{-1},\beta\alpha\beta^{-1},\beta))=
(\delta^2 (g_1)) (\alpha\beta\alpha^{-1}\beta^{-1},\beta\alpha\beta^{-1},\beta)\\
&=&g_1(\beta\alpha\beta^{-1},\beta)\, g_1(\alpha,\beta)^{-1}g_1(\alpha\beta\alpha^{-1}\beta^{-1},\beta\alpha)\, g_1(\alpha\beta\alpha^{-1}\beta^{-1},\beta\alpha\beta^{-1})^{-1},
\end{eqnarray*}
It follows from the definitions that
\begin{eqnarray*}
{\text {Tr}}(\mu_c\varphi_\beta)=g_1(\beta,\alpha)^{-1}g_1(\beta\alpha^{-1}\beta^{-1},\beta) \, g_1(\alpha\beta\alpha^{-1}\beta^{-1},\beta\alpha^{-1}\beta^{-1})^{-1}
\end{eqnarray*}
and
\begin{eqnarray*}
{\text {Tr}}(\varphi_{\alpha}^{-1}\mu_c)=g_1(\alpha\beta\alpha^{-1},\alpha)^{-1}g_1(\alpha,\beta)\,  g_1(\alpha\beta\alpha^{-1}\beta^{-1},\beta)^{-1}.
\end{eqnarray*}
Comparing these expressions, we obtain that
$$g_2(\kappa(\alpha\beta\alpha^{-1}\beta^{-1},\beta,\alpha)) \,
{\text {Tr}}(\mu_c\varphi_\beta)=
g_2(\kappa(\alpha\beta\alpha^{-1}\beta^{-1},\beta\alpha\beta^{-1},\beta))\,
{\text {Tr}}(\varphi_{\alpha}^{-1}\mu_c).$$ This formula is equivalent to (\ref{tr}).
\end{proof}

\begin{lemma} \label{example+}
Any simple TF-algebra $V=\oplus_{\alpha \in G} \, V_\alpha$ over
$(G,A,\kappa)$  such that  $\eta_{\varepsilon}(1_V\otimes 1_V)=1$ is
isomorphic to $V(g_1,g_2)$ for a certain  $\kappa$-pair $(g_1,
g_2)$. This $\kappa$-pair is determined by $V$ uniquely up to
multiplication of $g_1$ by   $  \delta^1(\psi)$  for a map $\psi
:G\to K^*$.
\end{lemma}
\begin{proof}
   For each $\alpha\in G$,  fix a nonzero vector
$l_{\alpha}\in V_{\alpha}$. In the role of  $l_\varepsilon\in
V_\varepsilon$ we take $1_V$. For any $a\in A$ and $\alpha\in G$, we
have
 $al_{\alpha}=g_2^{\alpha}(a)l_{\alpha}$ with $ g_2^{\alpha}(a)\in
 K$.
Clearly,
$$
l_{\alpha} = 1_A l_{\alpha}= (a^{-1}a)\, l_{\alpha}= g_2^{\alpha}  (a^{-1}) \, g_2^{\alpha}(a) \, l_{\alpha}
$$
and so  $g_2^{\alpha}(a)\in K^*$.

Given $\alpha, \beta \in G$, we have  $
 l_{\alpha}  l_{\beta} =c(\alpha,\beta)\, l_{\alpha\beta}
$ for some $c(\alpha,\beta)\in K$. Since the pairing \eqref{nas} is
non-degenerate, $ l_{\alpha}  l_{\alpha^{-1}} \neq 0$. Thus,
$c(\alpha,\alpha^{-1})\neq 0$ for all $\alpha$.

We claim that  $g_2^{\alpha}(a)$ does not depend on $\alpha$ for
every $a\in A$. Indeed,
\begin{eqnarray*}
(al_{\alpha})\, l_{\alpha^{-1}}=g_2^{\alpha}(a)\, l_{\alpha}\, l_{\alpha^{-1}}=g_2^{\alpha}(a)\,
c(\alpha, \alpha^{-1}) \, l_{\varepsilon}
\end{eqnarray*}
and
\begin{eqnarray*}
a(l_{\alpha}l_{\alpha^{-1}})=a c(\alpha,\alpha^{-1}) \, l_{\varepsilon}
=g_2^{\varepsilon}(a) c(\alpha, \alpha^{-1})\, l_{\varepsilon}.
\end{eqnarray*}
 Since $ (al_{\alpha})\,
l_{\alpha^{-1}}=a(l_{\alpha}l_{\alpha^{-1}})$ and $c(\alpha,
\alpha^{-1})\neq 0$, we   have $
g_2^{\alpha}(a)=g_2^{\varepsilon}(a)$. Set $g_2=g_2^\varepsilon:A\to
K^*$. Since $V $ is an $A$-module, the map $g_2$ is a group
homomorphism. Formula (\ref{bimodule}) implies the identity
\eqref{lol}.

Given $\alpha, \beta\in G$, we have
\begin{eqnarray*}
c(\beta,\beta^{-1})\, l_{\alpha}&=&c(\beta,\beta^{-1})\, l_{\alpha} \, l_\varepsilon \\
&=& l_{\alpha}(l_{\beta}\, l_{\beta^{-1}}) \\
&=&\kappa(\alpha,\beta,\beta^{-1})^{-1} (l_{\alpha}\, l_{\beta})\, l_{\beta^{-1}} \\
&=&\kappa(\alpha,\beta,\beta^{-1})^{-1} c(\alpha,\beta)\, l_{\alpha\beta}\, l_{\beta^{-1}}.
\end{eqnarray*}
Therefore $c(\alpha,\beta)\in K^*$. Let $g_1:G\times G\to K^*$ be
the map defined by $g_1 (\alpha, \beta)=(c(\alpha, \beta))^{-1}$ for
all $\alpha, \beta$. The identity $
l_{\alpha}l_{\beta}=g_1(\alpha,\beta)^{-1}l_{\alpha\beta} $ and
Formula (\ref{associativity}) yield
\begin{eqnarray*}
g_1(\alpha,\beta)^{-1}g_1(\alpha\beta,\gamma)^{-1}l_{\alpha\beta\gamma}=g_2(\kappa(\alpha,\beta,\gamma))\,
g_1(\alpha,\beta\gamma)^{-1}g_1(\beta,\gamma)^{-1}l_{\alpha\beta\gamma}.
\end{eqnarray*}
This implies that $g_2\circ \kappa= \delta^2 (g_1)$. The equality
\eqref{uuu} follows from the definitions and the choice
$l_\varepsilon=1_V$. Thus,  the maps  $g_1, g_2$ form a
$\kappa$-pair. It is clear that
  $V =V(g_1,g_2)$.

The second claim of the lemma   follows from the fact that any two
  bases  $l=\{l_\alpha\}_{\alpha\in G}$ and
$l'=\{l'_\alpha\}_{\alpha\in G}$ in $V$ as above are related by
$l_\alpha= \psi_\alpha l'_\alpha$, where  $\psi_\alpha\in K^*$ for
all $\alpha $. The bases $l$, $l'$ yield the same map  $g_2:A\to
K^*$ while the associated maps $g_1:G\times G\to K^*$    differ by
the coboundary of the map $G\to K^*, \alpha\to \psi_\alpha$.
\end{proof}

\subsection{The group $H^2(G,A,\kappa;K^*)$} The $\kappa$-pairs $(g_1:G\times G\to K^*, \, g_2:A\to K^*)$  form
an abelian group $H=H(G,A,\kappa)$ under pointwise multiplication.
The neutral element of $H$ is the $\kappa$-pair $(g_1=1, g_2=1)$.
The inverse of a $\kappa$-pair $(g_1, g_2)$ in $H$  is the
$\kappa$-pair $(g_1^{-1}, g_2^{-1})$. The coboundary $\kappa$-pairs
form a subgroup of $H$.  Let $H^2(G,A,\kappa;K^*)$ be the quotient
of $H$ by the subgroup of coboundary $\kappa$-pairs. Lemma
\ref{example+} yields   a bijective correspondence between  the set
$H^2(G,A,\kappa;K^*) \times K^*$ and the set of isomorphism classes
of simple TF-algebras over $(G, A, \kappa)$. This correspondence
assigns to a pair ($h\in H^2(G,A,\kappa;K^*)$, $z\in K^*$)  the
isomorphism class of the $z$-rescaled TF-algebra $V(g_1, g_2)$,
where $(g_1, g_2)$ is an arbitrary $\kappa$-pair representing $h$.
The following theorem due to Eilenberg and   MacLane computes
$H^2(G,A,\kappa;K^*)$    in topological terms.

\begin{theorem}\label{EMMM} (\cite{em}) Let   $X$ be a   path connected topological space with base point $x_0$.
Let $G=\pi_1(X,x_0)$, $A=\pi_2(X,x_0)$, and
$\Omega=\{\{{u}_{\alpha}\}_{\alpha\in G}; \{
 {f}_{\alpha, \beta}\}_{\alpha,\beta \in G}\}$  be a  basic system of loops and triangles in
 $X$.
 Let $\theta$ be a $K^*$-valued singular 2-cocycle on $X$ representing a cohomology class
 $[\theta]\in H^2(X,K^*)$.
 Then the pair  $(g_1:G\times G\to K^*, \,
g_2:A\to K^*)$   defined by
 $g_1(\alpha,\beta)=\theta(f_{\alpha,\beta})$ for   $\alpha, \beta\in G $ and  $ g_2(a)=\theta( {a})
 $ for   $a\in A$ is a $\kappa$-pair, where   $\kappa=\kappa^\Omega:G^3\to A$ be the 3-cocycle
 defined  in Section~\ref{XXX}. Moreover,   the  formula $[\theta] \mapsto (g_1, g_2)$ defines an isomorphism  $$H^2(X,K^*)\cong
H^2(G,A,\kappa;K^*).$$\label{cxk}
\end{theorem}

\subsection{The underlying TF-algebras of   cohomological
HQFTs} We keep the assumptions of Theorem~\ref{EMMM}.

\begin{theorem}
Let $(V^{\theta},\tau^{\theta})$ be the 2-dimensional $X$-HQFT
determined by $\theta\in H^2(X,K^*)$ (see \cite{tu}). If $\theta$
corresponds to  $(g_1,g_2)\in H^2(G,A,\kappa, K^*)$,  then the
underlying TF-algebra of $(V^{\theta},\tau^{\theta})$ is isomorphic
to  $V(g_1,g_2)$.
\end{theorem}
\begin{proof} Let      $V= \oplus_{\alpha\in G}\, V_\alpha= V^\Omega$
be the   underlying TF-algebra of $(V^{\theta},\tau^{\theta})$. The
definition of $ V^{\theta} $ implies that $V_\alpha=
V^{\theta}_{(S^1,u_{\alpha})} $ is a 1-dimensional $K$-vector space.
This vector space is generated by a vector $p_{\alpha} $ represented
by the map $p:[0,1]\to S^1$
   from
Section~2.2   viewed as a fundamental cycle of $
 S^1 $. Multiplication
 in  $V $ is
   computed by
$$p_{\alpha}\, p_{\beta}=\tau^{\theta}({{D}}(\alpha,\beta ))(p_{\alpha}\otimes p_{\beta})=f^*(\theta)(B)\, p_{\alpha\beta}, $$
 where $\alpha, \beta\in G$,  the map $f:{{D}}(\alpha,\beta)\to X$ is   determined by the structure of an $X$-surface in the
 disk with two  holes
 ${{D}}(\alpha,\beta)$,   and $B$ is a fundamental singular chain in ${{D}}(\alpha,\beta)$
such that $\partial (B)=p_{\alpha\beta}-p_{\alpha}-p_{\beta}$. It is
easy to see (cf.\ \cite{tu2}) that
$$g^*(\theta)(B)=\theta(f_{\alpha,\beta})^{-1}=g_1(\alpha,
\beta)^{-1}.$$ So,
 $p_{\alpha}\, p_{\beta}  =g_1(\alpha, \beta)^{-1}\,
 p_{\alpha\beta}$.
Similarly, for all $a\in A=\pi_2(X)$,
$$a \, p_{\varepsilon}=\tau^{\theta}({S}(\varepsilon,a))(p_{\varepsilon})=\theta( {a})\, p_{\varepsilon}=g_2( {a})\, p_{\varepsilon}.$$
where $\theta( {a})\in K^*$ is the evaluation  of   $\theta$ on $a
$.   Therefore $V =V(g_1,g_2)$.
\end{proof}

%\section*{Acknowledgment}

%%%%%%%%%%%%%%%%%%%%%%%%%%%%%
%%%%%%%%%%%%%%%%%%%%%%%%%%%%%%%%%%%%%%%
%%%%%%%%

%%%%%%%%%%

\bibliographystyle{amsalpha}

\end{document}